\documentclass[11pt]{article}
\usepackage{color}

\usepackage{amssymb}   
\usepackage{amsthm}    
\usepackage{amsmath}   
\usepackage{stmaryrd}  
\usepackage{titletoc}  
\usepackage{mathrsfs}  
\usepackage{graphicx}

\newlength{\defbaselineskip}
\newcommand{\setlinespacing}[1]%
           {\setlength{\baselineskip}{#1 \defbaselineskip}}

\theoremstyle{plain}
\newtheorem{thm}{Theorem}[section]
\newtheorem{cor}[thm]{Corollary}
\newtheorem{lem}[thm]{Lemma}
\newtheorem{prop}[thm]{Proposition}

\theoremstyle{definition}
\newtheorem{defn}{Definition}[section]

\newtheorem{rmk}{Remark}[section]

\newcommand{\eps}{\varepsilon}

\DeclareMathOperator*{\esssup}{esssup}
\DeclareMathOperator*{\essinf}{essinf}

\newcommand{\cL}{\mathcal{L}}

\newcommand{\cB}{\mathcal{B}}
\newcommand{\cA}{\mathcal{A}}
\newcommand{\cS}{\mathcal{S}}

\newcommand{\cG}{\mathcal{G}}

\newcommand{\cU}{\mathcal{U}}

\newcommand{\bH}{\mathbb{H}}
\newcommand{\bP}{\mathbb{P}}
\newcommand{\bR}{\mathbb{R}}
\newcommand{\bN}{\mathbb{N}}

\newcommand{\sF}{\mathscr{F}}
\newcommand{\sP}{\mathscr{P}}

\makeatletter\@addtoreset{equation}{section} \makeatother
 \allowdisplaybreaks
\begin{document}

\title{Stochastic Path-Dependent Hamilton-Jacobi-Bellman Equations
and Controlled Stochastic Differential Equations with Random Path-Dependent Coefficients\footnotemark[1] 
}

\author{Jinniao Qiu\footnotemark[2] 
}
\footnotetext[1]{This work was partially supported by the National Science and Engineering Research Council of Canada (NSERC) and by the start-up funds from the University of Calgary. }
\footnotetext[2]{Department of Mathematics \& Statistics, University of Calgary, 2500 University Drive NW, Calgary, AB T2N 1N4, Canada. \textit{E-mail}: \texttt{jinniao.qiu@ucalgary.ca} (J. Qiu)
.}
%
%

\maketitle

\begin{abstract}
In this paper, we propose and study the stochastic path-dependent Hamilton-Jacobi-Bellman (SPHJB) equation that arises naturally from the optimal stochastic control problem of stochastic differential equations with path-dependence and measurable randomness. Both the notions of viscosity solution and classical solution are proposed, and the value function of the optimal stochastic control problem is proved to be the viscosity solution to the associated SPHJB equation. A uniqueness result about viscosity solutions is also given for certain superparabolic cases, while the uniqueness of classical solution is addressed for general cases. In addition, an It\^o-Kunita-Wentzell-Krylov formula is proved for the compositions of random fields and stochastic differential equations in the path-dependent setting.
\end{abstract}



\section{Introduction}
Let $(\Omega,\sF,\{\sF_t\}_{t\geq0},\bP)$ be a complete filtered probability space on which the filtration $\{\sF_t\}_{t\geq0}$ satisfies the usual conditions and is generated by an $m$-dimensional Wiener process $W=\{W(t):t\in[0,\infty)\}$ together with all the $\bP$-null sets in $\sF$. The associated predictable $\sigma$-algebra  on $\Omega\times[0,T]$ is denoted by  $\sP$. Let $C([0,T];\bR^d)$ be the space of $\bR^d$-valued continuous functions on $[0,T]$.  For each $x\in C([0,T];\bR^d)$, denote by $x_t$ its restriction to time interval $[0,t]$ for each $t\in [0,T]$ and by $x(t)$ its value at time $t\in[0,T]$.

Consider the following stochastic optimal control problem
\begin{align}
\min_{\theta\in\cU}E\left[\int_0^T\!\! f(s,X_s,\theta(s))\,ds +G(X_T) \right] \label{Control-probm}
\end{align}
subject to
\begin{equation}\label{state-proces-contrl}
\left\{
\begin{split}
&dX(t)=\beta(t,X_t,\theta(t))\, dt + \sigma(t,X_t,\theta(t))\,dW(t),\,\,
\,t\geq 0; \\
& X_0=x_0\in\bR^d.
\end{split}
\right.
\end{equation}
Here and throughout this paper,  the number $T\in (0,\infty)$ denotes a fixed deterministic terminal time,  and  $\cU$ represents the set of all the $U$-valued and $\sF_t$-adapted processes with $U\subset \bR^{\bar m}$ ($\bar m\in \bN^+$) being a nonempty set. The state process $(X(t))_{t\in[0,T]}$, governed by the {\sl control} $\theta\in\cU$ may be written as $X^{r,x_r;\theta}(t)$ for $0\leq r\leq t\leq T$ to indicate the dependence of the state process on the control $\theta$, the initial time $r$ and initial path $x_r$.  

In this paper, we consider the non-Markovian case where the coefficients $\beta,\sigma, f$, and $G$ depend not only on time and control but also \textit{explicitly}  on $\omega \in\Omega$ and paths/history of the state process. 
Define the dynamic cost functional:\begin{align}
J(t,x_t;\theta)=E_{\sF_t}\left[\int_t^T\!\! f(s,X^{t,x_t;\theta}_s,\theta(s))\,ds +G(X^{t,x_t;\theta}_T) \right],\ \ t\in[0,T],
\label{eq-cost-funct}
\end{align}
where $E_{\sF_t}[\,\cdot\,]$ denotes the conditional expectation with respect to $\sF_t$. 
Then, the value function is given by
\begin{align}
V(t,x_t)=\essinf_{\theta\in\cU}J(t,x_t;\theta),\quad t\in[0,T].
\label{eq-value-func}
\end{align}
Due to the randomness and path-dependence of the coefficient(s), the value function $V(t,x_t)$ is generally a function of time $t$, path $x_t$, and $\omega\in\Omega$, and it proves to satisfy the following stochastic path-dependent Hamilton-Jacobi-Bellman (SPHJB) equation:
{\small
\begin{equation}\label{SHJB}
  \left\{\begin{array}{l}
  \begin{split}
  -\mathfrak{d}_t u(t,x_t)
- \mathbb{H}(t,x_t,\nabla u(t,x_t),\nabla^2u(t,x_t),\mathfrak{d}_{\omega}\nabla u(t,x_t) )&=0,\, (t,x)\in [0,T)\times  C([0,T];\bR^d);\\
    u(T,x)&= G(x), \quad x\in   C([0,T];\bR^d),
    \end{split}
  \end{array}\right.
\end{equation}
}
with
\begin{align*}
\mathbb{H}(t,x_t,p,A,B)
= \essinf_{v\in U} \bigg\{
\text{tr}\left(\frac{1}{2}\sigma \sigma'(t,x_t,v) A+\sigma(t,x_t,v) B\right)
       +\beta'(t,x_t,v)p +f(t,x_t,v)
                \bigg\},
\end{align*}
for $(p,A,B)\in \bR^d\times\bR^{d\times d}\times \bR^{m\times d}$.
 Here, $\nabla u(t,x_t)$ and $\nabla^2u(t,x_t)$  represent, respectively, the first and second order vertical derivative of $u(t,x_t)$ at the path $x_t$ (see Definition \ref{defn-VD}) and  the unknown \textit{adapted} random field $u$ is is confined to the following form:
  \begin{align}
u(t,x_t)=u(T,x_{t,T-t})-\int_{t}^T\mathfrak{d}_{s} u(s,x_{t,s-t})\,ds-\int_t^T\mathfrak{d}_{w}u(s,x_{t,s-t})\,dW(s), \label{SDE-u}
\end{align}
where $x_{t,r-t}(s)=x_t(s) {1}_{[0,t)}(s) + x_t(t){1}_{[t,r]}(s)$ for $0\leq t\leq s\leq r\leq T$.
The Doob-Meyer decomposition theorem indicates the uniqueness of the pair $(\mathfrak{d}_tu,\,\mathfrak{d}_{\omega}u)$ and thus the linear operators $\mathfrak{d}_t$ and $\mathfrak{d}_{\omega}$ are well defined in certain spaces (see Definition \ref{defn-testfunc}).  The pair $(\mathfrak{d}_tu,\,\mathfrak{d}_{\omega}u)$ may also be defined as two differential operators; see \cite[Section 5.2]{cont2013-founctional-aop} and \cite[Theorem 4.3]{Leao-etal-2018} for instance. 

When it holds that 
\begin{align}
\mathfrak{d}_{\omega}\nabla u (t,x_t)= \nabla \mathfrak{d}_{\omega} u(t,x_t), \quad \text{a.s. for all }
t\in[0,T) \text{ and }x\in C([0,T];\bR^d),
\label{der-eq}
\end{align}
putting $\psi=\mathfrak{d}_{w}u$ and comparing \eqref{SHJB} and \eqref{SDE-u}, we may rewrite the SPHJB equation \eqref{SHJB} as:
\begin{equation*}
  \left\{\begin{array}{l}
  \begin{split}
  -d u(t,x_t)
&= \mathbb{H}(t,x_t,\nabla u(t,x_t),\nabla^2u(t,x_t), \nabla \psi(t,x_t) )\,dt-\psi(t,x_t)\,dW(t);\\
    u(T,x)&= G(x), \quad x\in   C([0,T];\bR^d),
    \end{split}
  \end{array}\right.
\end{equation*}
which turns out to be a fully nonlinear backward stochastic partial differential equation (BSPDE), nevertheless, defined on path spaces. In fact, the relation \eqref{der-eq} holds true when all the coefficients are just state-dependent, i.e., with probability one, $(\beta,\sigma,f)(t,x_t,v)=(\beta,\sigma,f)(t,x(t),v)$, and $G(x_T)=G(x(T))$ for all $(t,x,v)\in [0,T]\times C([0,T];\bR^d)\times U$; see \cite{qiu2017viscosity} for instance. This sheds light on the connections between SPHJB equation \eqref{SHJB} and the BSPDEs; for related research on general BSPDEs, we refer to \cite{Bayraktar-Qiu_2017,cardaliaguet2015master,DuTangZhang-2013,EnglezosKaratzas09,Hu_Ma_Yong02,Peng_92} among many others. However, the exchangeability \eqref{der-eq}  generally does not hold due to the path-dependence (see Remark \ref{rmk-derivative} for examples), which makes SPHJB equation \eqref{SHJB} stand beyond the realm of BSPDEs.

When  all the coefficients $\beta,\sigma,f,$ and $G$ are \textit{deterministic continuous} path-dependent functions of time $t$, control $\theta$, and the paths of $(X,W)$,  the SPHJB equation \eqref{SHJB} falls into the range of so-called fully nonlinear path-dependent partial differential equations (PPDEs); refer to \cite{cont2013-founctional-aop,cosso2018path,ekren2016viscosity-1,ekren2016viscosity-2,ekren2016pseudo,peng2016bsde,ren2017comparison,tang2012path} to mention just a few. The viscosity solution theory of such PPDEs involves admirable path-dependent calculus. However, such PPDEs are deterministic and due to certain continuity requirements on paths of $(X,W)$, the viscosity solutions fail to incorporate the conventional $L^2$-theory of backward stochastic differential equations (BSDEs); for instance, the following trivial BSDE:
{\small
$$
Y(t)= \xi-\int_t^T Z(t)\,dW(t), \quad t\in[0,T]; \quad \xi\in L^{\infty}(\Omega,\sF_T;\bR),
$$
}
 by the martingale representation theorem, admits a unique $L^2$-solution pair $(Y,Z)$ with $Y(t,\omega)=E_{\sF_t}[\xi] (\omega)$, which does not require the continuity in $\omega$ of the solution $(Y,Z)(t,\omega)$ or the given terminal value $\xi\in L^{\infty}(\Omega,\sF_T;\bR)$. This observation motivates our considerations of measurable randomness and path-dependence and in fact gives a nontrivial meaning to the proposed  SPHJB equations and associated solution theory with different methods.  

In this paper, we propose the SPHJB equation \eqref{SHJB} for the optimal stochastic control problem \eqref{Control-probm}. Both classical solutions and viscosity solutions are discussed. The value function $V$ is verified to be a viscosity solution. A uniqueness result about viscosity solutions is addressed for the superparabolic cases with state-dependent $\sigma$, and as a byproduct, the uniqueness of classical solution is derived for general cases. In addition, an It\^o-Kunita-Wentzell-Krylov formula is proved for the compositions of random fields and stochastic differential equations in a path-dependent setting. 

Due to the  mixture of path-dependence and measurable randomness,  a viscosity solution theory for SPHJB equation \eqref{SHJB} is never a trivial task. On the one hand, due to the path-dependence, the solution $u(\omega,t,\cdot)$, for each $(\omega,t)\in\Omega\times [0,T]$, is path-wisely defined on the path space $C([0,T];\bR^d)$, and we have to deal with the lack of local compactness of the path space; instead of using  the nonlinear expectation techniques via second order BSDEs for deterministic PPDEs (see \cite{ekren2014viscosity,ekren2016viscosity-1} for instance), we define the random test functions by taking extreme points in certain locally compact subspaces (actually H\"older spaces) via conventional optimal stopping times. On the other hand,  as the involved coefficients are  just measurable w.r.t. $\omega$ on the sample space $(\Omega,\sF)$ without any specified topology,  it is not appropriate to define the viscosity solutions in a pointwise manner w.r.t. $\omega\in (\Omega,\sF)$; instead, we use a class of random fields of form \eqref{SDE-u} having sufficient spacial regularity as test functions; at each point $(\tau,\xi)$ ($\tau$ may be stopping time and $\xi$ a $C([0,\tau];\bR^d)$-valued $\sF_{\tau}$-measurable variable) the classes of test functions are also parameterized by the measurable sets $\Omega_{\tau}\in\sF_{\tau}$ and the type of compact subspaces. 

Finally, we compare the present work with the accompanying one \cite{qiu2020controlled}. In fact, when $\sigma(\omega,t,x_t,\theta(t))(\equiv \sigma(\omega,t))$ is path-invariant and uncontrolled in \eqref{state-proces-contrl}, we may take $\overline X(t)=X(t)-\xi(t)$ with $\xi(t)=\int_0^t\sigma(s)\,dW(s)$ for $t\in [0,T]$, and then the optimization \eqref{Control-probm}-\eqref{state-proces-contrl} is equivalent to the following one:
\begin{align}
\min_{\theta\in\cU}E\left[\int_0^T\!\! f(s,(\overline X+\xi)_s,\theta(s))\,ds +G((\overline X+\xi)_T) \right], \label{Control-probm-comment}
\end{align}
subject to
\begin{equation}\label{state-proces-contrl-comment}
\left\{
\begin{split}
&
\frac{d\overline X(t)}{dt}=\beta(t,(\overline X+\xi)_t,\theta(t)),\,\,
\,t\geq 0; \quad
\\
& 
\overline X_0=x_0\in\bR^d.
\end{split}
\right.
\end{equation}
The paper \cite{qiu2020controlled} is devoted to the control problem \eqref{Control-probm-comment}-\eqref{state-proces-contrl-comment} and  the existence and uniqueness of viscosity solution is addressed for the associated stochastic path-dependent Hamilton-Jacobi equation which, we note, 
%
is first-order. In contrast, our SPHJB \eqref{SHJB} is second-order, and this leads to the different methods and contents for the viscosity solution theory. For instance, to deal with the lacking of local compactness of the path space, subspaces of Lipschitz functions are used for treating viscosity solutions in \cite{qiu2020controlled}, while we use subspaces of H\"older functions herein because of the controlled stochastic integrals in the state process \eqref{state-proces-contrl}. Two Lipschitz functions over two successive time intervals with a joint point and an identical Lipschitz constant may be pieced together as a new Lipschitz function with the same Lipschitz constant, which, however, does not hold for H\"older functions. This together with the lacking of boundedness estimates of the second-order terms $\nabla^2 u$ and $\mathfrak{d}_{\omega}\nabla u$ for SPHJB \eqref{SHJB} gives rise to, particularly, the different discussions on the uniqueness of viscosity solutions.

The rest of this paper is organized as follows. In Section 2, we introduce some notations, show the standing assumption on the coefficients, and define both the viscosity (semi)solutions and classical (semi)solutions. In Section 3, a generalized It\^o-Kunita-Wentzell-Krylov formula is proved in a path-dependent setting and then it is applied to semisolutions. Section 4 is devoted to the existence of the viscosity solution, while in Section 5,  we discuss the uniqueness. 
%


\section{Preliminaries and definition of solutions}

\subsection{Preliminaries}
  For each $(k,r)\in\bN^+\times [0,T]$, denote by $\Lambda_r^0(\bR^k):=C([0,r];\bR^k)$ the space of all $\bR^k$-valued continuous functions on $[0,r]$ and by $\Lambda_r(\bR^k):=D([0,r];\bR^k)$  the space of $\bR^k$-valued c\`{a}dl\`ag (right continuous with left limits) functions on $[0,r]$. Set
$$
\Lambda^0(\bR^k)=\cup_{r\in[0,T]} \Lambda^0_r(\bR^k),\quad \Lambda(\bR^k)=\cup_{r\in[0,T]} \Lambda_r(\bR^k).
$$
For each path $X\in \Lambda_T(\bR^k)$ and $t\in[0,T]$, let $X_t=(X(s))_{0\leq s\leq t}$ be its restriction to time interval $[0,t]$,  and $X(t)$ its value at time $t$. When $k=d$, we write $\Lambda^0$, $\Lambda_r^0$, $\Lambda$, and $\Lambda_r$ for simplicity.

Both $\Lambda$ and $\Lambda^0$ are endowed with the following quasi-norm and metric: for each $(x_r,\bar x_t)\in \Lambda_r\times \Lambda_t$ or $(x_r,\bar x_t)\in  \Lambda^0_r\times \Lambda^0_t$ with $0\leq r\leq t\leq T$,
\begin{align*}
\|x_r\|_{0}&= \sup_{s\in[0,r]} |x_r(s)|;
\\
d_0(x_r,\bar x_t)&=\sqrt{|t-r|} + \sup_{s\in [0,t]} \left\{ |x_r(s)-\bar x_t(s)| 1_{[0,r)}(s)+ |x_r(r)-\bar x_t(s)| 1_{[r,t]}(s) \right\}.
\end{align*}
Then both $(\Lambda^0_t, \|\cdot\|_0)$ and $(\Lambda_t, \|\cdot\|_0)$ are Banach spaces for each $t\in[0,T]$, while $(\Lambda^0,  d_0)$ and $(\Lambda, d_0)$ are complete metric spaces. In fact, for each $t\in[0,T]$, $(\Lambda^0_t, \|\cdot\|_0)$ and $(\Lambda_t, \|\cdot\|_0)$ can be and (throughout this paper) will be thought of as the complete subspaces of $(\Lambda^0_T, \|\cdot\|_0)$ and $(\Lambda_T, \|\cdot\|_0)$, respectively; indeed, for each $x_t\in \Lambda_t$ ($x_t\in \Lambda_t^0$, respectively), we define, correspondingly, $\bar x\in\Lambda_T$ ($\bar x\in\Lambda^0_T$, respectively) with $\bar x (s)=x_t(t\wedge s)$ for $s\in[0,T]$. In addition, we shall use $\cB(\Lambda^0)$, $\cB(\Lambda)$, $\cB(\Lambda^0_t)$ and $\cB(\Lambda_t)$ to denote the corresponding Borel $\sigma$-algebras. By contrast, for each $\delta>0$ and $x_r\in\Lambda$, denote by $B_{\delta}(x_r)$ the set of paths $y_t\in\Lambda$ satisfying $d_0(x_r,y_t)\leq \delta$.

 For each $(x_t,h)\in \Lambda_t\times \bR^d$, its vertical perturbation is given as $x^h_t \in \Lambda_t$ with
 $x_t^h(s)= x_t(s) 1_{[0,t)}(s)+ \left(x_t(t)+h\right)1_{\{t\}}(s)$ for $s\in[0,t]$. 
 \begin{defn}\label{defn-VD}
 Given a functional $\phi$: $\Lambda \rightarrow \bR$ and a path $x_t\in \Lambda_t$, we say that $\phi$ is differentiable at $x_t$ if the function 
 $
 \phi(x_t^{\cdot}):\,
 \bR^d\rightarrow \bR,
\quad h\mapsto  \phi(x_t^h)
$
 is differentiable at $0$. The gradient  
 $$
 \nabla \phi(x_t):=(\nabla_1\phi(x_t),\dots, \nabla_d\phi(x_t))' 
 \quad\text{with}\quad 
 \nabla_i \phi(x_t):=\lim_{\delta\rightarrow 0}\frac{\phi(x_t^{\delta e_i}) -\phi(x_t)}{\delta}
 $$
  is called the vertical derivative of $\phi$ at $x_t$, where $\{e_i\}_{i=1,\dots,d}$ is the canonical basis in $\bR^d$.
 \end{defn}

Let $(\mathbb B,\|\cdot\|_{\mathbb B }) $ be a Banach space. 
If the $\mathbb B$-valued functional $\phi$ is continuous and bounded at all $x_t\in\Lambda$, $\phi$ is said to be continuous on $\Lambda$ and denoted by $\phi\in C(\Lambda; \mathbb B)$. 
Similarly, we define $C(\Lambda^0;\mathbb B)$, $C([0,T]\times\Lambda;\mathbb B)$, and $C([0,T]\times\Lambda^0;\mathbb B)$. In particular, we define $C((a,b] \times\Lambda;\mathbb B)= \cap_{\delta\in (0,b-a)} C([a+\delta,b]\times\Lambda;\mathbb B)
$ as usual for $0\leq a<b$, and so is it for $C((a,b)\times\Lambda;\mathbb B)$.

 For each $t\in[0,T]$, let $L^0(\Omega\times\Lambda_t,\sF_t\otimes\cB(\Lambda_t);\mathbb B)$ be the space of $\mathbb B$-valued $\sF_t\otimes\cB(\Lambda_t)$-measurable random variables. The measurable function
 $$
 u: \quad (\Omega\times[0,T]\times \Lambda,\, \sF\otimes\cB([0,T])\otimes \cB(\Lambda)  )   \rightarrow (\mathbb B,\,\cB(\mathbb B)),
 $$
   is said to be \textit{adapted} if for any time $t\in [0,T]$, $u$ is $\sF_t\otimes \cB(\Lambda_t)$-measurable.  
 For $p\in[1,\infty]$, denote by $\cS ^p (\Lambda; {\mathbb B })$ the set of all the adapted functions $u$: $\Omega\times[0,T]\times \Lambda\rightarrow \mathbb B$ such
 that for almost all $\omega\in\Omega$, $u$ is valued in $C([0,T]\times\Lambda;\mathbb B)$
 and
{\small $$\| u\|_{\cS ^p(\Lambda; {\mathbb B })}:= \left\| 
\sup_{(t,x_t)\in [0,T]\times \Lambda_t} \|u(t,x_t)\|_{\mathbb B } \right\|_{L^p(\Omega,\sF,\bP)}< \infty.$$
}
 For $p\in [1,\infty)$, denote by $\mathcal{L}^p(\Lambda; {\mathbb B })$ the set of all  the adapted functions $\mathcal \psi$: $\Omega\times[0,T]\times \Lambda\rightarrow \mathbb B$ such
 that for almost all $(\omega,t)\in\Omega\times[0,T]$, $\psi(t)$ is valued in $C(\Lambda_t;\mathbb B)$, and 
{\small
 $$
 \|\psi\|_{\mathcal{L}^p(\Lambda; {\mathbb B })}:=\left(E \left[ \int_0^T \sup_{x_t\in\Lambda_t}\|\psi(t,x_t)\|_{\mathbb B }^p\,dt\right] \right)^{1/p}< \infty.
 $$
 }
Both $(\cS^p(\Lambda; {\mathbb B }),\,\|\cdot\|_{\cS^p(\Lambda; {\mathbb B })})$ and $(\mathcal{L}^p(\Lambda; {\mathbb B }),\|\cdot\|_{\mathcal{L}^p(\Lambda; {\mathbb B })})$
are Banach spaces. Analogously, we define $L^0(\Omega\times\Lambda^0_t,\sF_t\otimes\cB(\Lambda^0_t);\mathbb B)$, $(\cS^p(\Lambda^0; {\mathbb B }),\,\|\cdot\|_{\cS^p(\Lambda^0; {\mathbb B })})$, and $(\mathcal{L}^p(\Lambda^0; {\mathbb B }),\|\cdot\|_{\mathcal{L}^p(\Lambda^0; {\mathbb B })})$.

 As usual, we use $C$ with or without a subscript to denote a constant whose value may vary from line to line. Throughout this paper, we use the following assumption.\\[3pt]
\noindent
   $({\mathcal A} 1)$ \it $G\in L^{\infty}(\Omega,\sF_T;C(\Lambda_T;\bR))$. For the coefficients $g=f,\beta^i,\sigma^{ij}$, $i=1,\dots,d$, $j=1,\dots,m$, \\
(i) 
 for each $v\in U$, $g(\cdot,\cdot,v)$ is adapted;\\
(ii) for almost all $(\omega,t)\in\Omega\times [0,T]$, $g(t,\cdot,\cdot)$ is continuous on $\Lambda_t\times U$;\\
(iii) there exists $L>0$ such that for all $x,\bar x\in \Lambda_T$, $t\in[0,T]$ and $\gamma_t,\bar\gamma_t\in \Lambda_t$, there hold
\begin{align*}
\esssup_{\omega\in\Omega} |G(x)|+ 
\esssup_{\omega\in\Omega} \sup_{v\in U} |g(t,\gamma_t, v) |
& \leq L ,\\
\esssup_{\omega\in\Omega} |G(x)-G(\bar x)|+ 
\esssup_{\omega\in\Omega} \sup_{v\in U} |g(t,\gamma_t, v)-g(t,\bar \gamma_t,v)|
& \leq L\left(\|x-\bar x\|_0 + \|\gamma_t-\bar\gamma_t\|_0   \right).
\end{align*}
\rm



\subsection{Definition of the solutions}



For $\delta \geq 0$ and $x_t\in \Lambda_t$, the horizontal extension $x_{t,\delta} \in \Lambda_{t+\delta}$  is given as $x_{t,\delta}(s)=x_t(s\wedge t)$ for $s\in[0,t+\delta]$.

\begin{defn}\label{defn-testfunc}
For $u\in \cS^{2} (\Lambda;\bR)$ with $\nabla u\in \cL^2(\Lambda;\bR)$, we say $u\in \mathscr C_{\sF}^2$ if 
there exist a constant $\alpha\in (0,1)$ and a finite partition $0=\underline t_0<\underline t_1<\ldots<\underline t_n=T$, for integer $n\geq 1$, such that 
\begin{enumerate}
\item [(i)] on each subinterval $[\underline t_j,\underline t_{j+1})$, $j=0,\ldots, n-1$, 
\begin{enumerate}
\item  there exists $(\mathfrak{d}_tu, \,\mathfrak{d}_{\omega}u)$ with 
$$
(\mathfrak{d}_t u, \,\mathfrak{d}_{\omega} u) 1_{\{[\underline t_j,\underline t_{j+1}-\eps]\}} \in \cL^2(\Lambda;\bR^d) \times  \cL^2(\Lambda;\bR^{ m}),\quad \forall\, \eps \in (0, \underline t_{j+1} -\underline t_j),
$$ 
%
satisfying for all $\underline t_j \leq r \leq \tau< \underline t_{j+1}$, and all $x_r\in \Lambda_r$,
\begin{align*}
 u(\tau,x_{r,\tau-r})=u( r,x_{r}) + \int_{r}^{\tau} 
 	\mathfrak{d}_s u(r,x_{r,s-r})\,ds 
+\int^{\tau}_r\mathfrak{d}_{\omega}  u(s,x_{r,s-r})\,dW(s),\text{ a.s.;}
\end{align*}
\item [(b)] $\nabla u$ is a.s. valued in $C((\underline t_j,\underline t_{j+1} )\times \Lambda;\bR^d)$, and there exists some  adapted $C((\underline t_j,\underline t_{j+1} )\times \Lambda;\bR^{m\times d})$-valued function denoted by $ \mathfrak{d}_{\omega} \nabla u $ such that for each subinterval $[\tilde t_j,\tilde t_{j+1}] \subset  ( \underline t_j,\underline t_{j+1}) $, $X\in \Lambda_T$, and any $M^l(t)=\int_{t\wedge \tilde t_j}^t g(s)\,dW^l(s)$ for some $g\in L^{\infty}(\Omega\times [0,T];\sP) $, $t\in [\tilde t_j,\tilde t_{j+1}]$, there holds
\begin{align}
\left\langle\nabla_i u(\cdot,X_{\cdot}),\, M^l(\cdot)  \right\rangle_{\tilde t_j}^{\tilde t_{j+1}}
=\int_{\tilde t_j}^{\tilde t_{j+1}} \left(\mathfrak{d}_{\omega} \nabla u\right)^{li} (t,X_t) g(t)\,dt, \quad\text{a.s.,}
\label{def-dwD} 
\end{align}
for $i=1,\ldots,d$, $l=1,\ldots,m$, where the covariation (bracket) of two stochastic processes is defined as usual:
{\small
\begin{align*} 
&\left\langle\nabla_i u(\cdot,X_{\cdot}),\, M^l(\cdot)  \right\rangle_{\tilde t_j}^{\tilde t_{j+1}}
\\
&=\lim_{|\Pi| \rightarrow 0^+} \sum_{k=0}^{N-1}  \left( \nabla_i u(\tau_{k+1},X_{\tau_k,\tau_{k+1}-\tau_k})- \nabla_i u(\tau_k,X_{\tau_k})\right) \int_{\tau_k}^{\tau_{k+1}} g(s)\,dW^l(s),\text{ in probability,}
\end{align*}
}
with $\Pi=\{\tilde t_j= \tau_0<\ldots<\tau_N=\tilde t_{j+1}\}$ being a subdivision of $[\tilde t_j,\tilde t_{j+1}]$ and $|\Pi|=\max_{1\leq k\leq N}|\tau_k-\tau_{k-1}|$;
\end{enumerate}
\item [(ii)] for each $0<\delta<\max_{0\leq j \leq n-1} |\underline t_{j+1}-\underline t_j|$, and  $g=\mathfrak{d}_tu,\,\nabla_i u, \,\nabla_{ij}u, \,(\mathfrak{d}_{\omega}u)^j, (\mathfrak{d}_{\omega}\nabla u)^{ji}$, $i=1,\dots,d$, $j=1,\dots,m$, there exists $L^{\delta}_{\alpha} \in (0,\infty) $ satisfying a.s. for all 
$t\in \cup_{0\leq j \leq n-1} (\underline t_j, \underline t_{j+1}-\delta]$ and all $x_t,y_t\in \Lambda_t$,
\begin{align*}
|\nabla u(t,x_t)| + |\nabla^2 u(t,x_t)|  +|\mathfrak{d}_{\omega}\nabla u(t,x_t)| 
&
	\leq L^{\delta}_{\alpha},
\\
|g(t,x_t)-g(t,y_t)| 
&
	\leq L^{\delta}_{\alpha} \left\|x_t-y_t\right\|_0^{\alpha}.
\end{align*}
\end{enumerate}
 We say the number $\alpha$ is the exponent associated to $u\in \mathscr C^2_{\sF}$ and $0=\underline t_0<\underline t_1<\ldots<\underline t_n=T$ the associated partition.\footnote{The exponent $\alpha$ is not put in the notation $\mathscr C^{2}_{\sF}$, as in many applications, there is no need to specify the exponent.}
\end{defn}

%

Doob-Meyer decomposition theorem gives the uniqueness of the pair $(\mathfrak{d}_tu,\,\mathfrak{d}_{\omega}u)$ at $(\omega, t,x_{s,t-s})$ for $0\leq s <t\leq T$, and with a standard denseness argument we may define the pair $(\mathfrak{d}_tu,\,\mathfrak{d}_{\omega}u)$ in $ \cL^2(\Lambda;\bR)\times  \cL^2(\Lambda;\bR^m)$ with 
$$
(\mathfrak{d}_tu,\,\mathfrak{d}_{\omega}u)(t,x_t)
=\lim_{s\rightarrow t^-}(\mathfrak{d}_tu,\,\mathfrak{d}_{\omega}u)(t,x_{s,t-s})=(\mathfrak{d}_tu,\,\mathfrak{d}_{\omega}u)(t,\lim_{s\rightarrow t^-} x_{s,t-s}),\quad \forall\, x_t\in \Lambda_t.
$$
 This makes sense of the two linear operators $\mathfrak{d}_t$ and $\mathfrak{d}_{\omega}$ which are consistent with the differential operators  in \cite[Section 5.2]{cont2013-founctional-aop} and \cite[Theorem 4.3]{Leao-etal-2018}. In particular, if the random function $u$ on $\Omega\times[0,T]\times \bR^d$ is regular enough (w.r.t. $\omega$), the term $\mathfrak{d}_{\omega}u$ is just the Malliavin derivative; if $u(t,x)$ is a deterministic function on the time-\textit{state} space $[0,T]\times\bR^d$, we may have $\mathfrak{d}_{\omega}u \equiv 0$ and $\mathfrak{d}_tu$ coincides with the classical derivative in time.  Noteworthily, the operators $\mathfrak{d}_t$ and $\mathfrak{d}_{\omega}$ are different from the path derivatives $(\partial_t,\,\partial_{\omega})$ via the functional It\^o formulas (see \cite{buckdahn2015pathwise} and \cite[Section 2.3]{ekren2016viscosity-1}); if $u(\omega,t,x_t)$ is smooth enough w.r.t. $(\omega,t)$ in the path space,  for each $x\in \Lambda^0_T$,  we have the relation 
$$\mathfrak{d}_tu(\omega,t,x_{s,t-s})=\left(\partial_t+\frac{1}{2}\partial^2_{\omega\omega}\right)u(\omega,t,x_{s,t-s}),\quad \mathfrak{d}_{\omega}u(\omega,t,x_{s,t-s})=  \partial_{\omega}u(\omega,t,x_{s,t-s}),$$
for $0\leq s<t<T$,     which may be seen from   \cite[Section 6]{ekren2016viscosity-1} and \cite{buckdahn2015pathwise}.  

By \eqref{def-dwD},  we define $\mathfrak{d}_{\omega} \nabla u$ in a way different from $\mathfrak{d}_{\omega} u$. Indeed, if there is $(\mathfrak{d}_t\nabla u, \,\mathfrak{D}_{\omega}\nabla u)$ with 
$$
(\mathfrak{d}_t\nabla u, \,\mathfrak{D}_{\omega}\nabla u) 1_{\{[\underline t_j,\, \underline t_{j+1} -\eps]\}} \in \cL^2(\Lambda;\bR^d) \times  \cL^2(\Lambda;\bR^{ m\times d}),\quad \forall\, \eps \in (0, \underline t_{j+1} -\underline t_j),
$$ 
satisfying for all $\underline t_j \leq r \leq \tau< \underline t_{t+1}$, and $x_r\in \Lambda_r$,
\begin{align}
 \nabla u(\tau,x_{r,\tau-r})=\nabla u( r,x_{r}) + \int_{r}^{\tau} 
 	\mathfrak{d}_s \nabla u(r,x_{r,s-r})\,ds 
+\int^{\tau}_r (\mathfrak{D}_{\omega}  \nabla u)'(s,x_{r,s-r})\,dW(s),\text{ a.s.,} \label{def-dwD-1}
\end{align}
one may easily check that $\mathfrak{d}_{\omega} \nabla u =   \mathfrak{D}_{\omega}  \nabla u$ which unveils the consistency. We do not adopt the method via \eqref{def-dwD-1} to define $\mathfrak{d}_{\omega} \nabla u$, basically because it requires the existence of $\mathfrak{d}_s \nabla u$ that is not necessary but narrows the test function space $ \mathscr C_{\sF}^2$.

\begin{rmk}\label{rmk-derivative}
It is worth noting that the operators $\mathfrak{d}_{\omega}$ and $\nabla $ are not exchangeable in many cases. For instance, taking $d=m=1$, $u(t,x_t)=\int_0^t\sin(x_t(s))\,dW(s)$, one has $\mathfrak{d}_{\omega} u(t,x_t)= \sin(x_t(t-)) $, $\nabla u(t,x_t) = 0$; however, $\nabla \mathfrak{d}_{\omega} u(t,x_t) =\cos(x_t(t-)) $, while $ \mathfrak{d}_{\omega} \nabla u(t,x_t) = 0 $.
\end{rmk}

For each stopping time $t\leq T$, let $\mathcal{T}^t$ be the set of stopping times $\tau$ valued in $[t,T]$ and $\mathcal{T}^t_+$ the subset of $\mathcal{T}^t$ such that $\tau>t$ for each $\tau\in \mathcal{T}^t_+$. For each $\tau\in\mathcal T^0$ and $\Omega_{\tau}\in\sF_{\tau}$, we denote by $L^0(\Omega_{\tau},\sF_{\tau};\Lambda^0_{\tau})$ the set of $\Lambda^0_{\tau}$-valued $\sF_{\tau}$-measurable functions.

Given a Banach space $(\mathbb B, \|\cdot\|_{\mathbb B})$, for each $\alpha\in (0,1)$ and  $0\leq t_0 <t_1\leq T$, denote by $C^{\alpha}([t_0,t_1];\mathbb B)$ the usual $\alpha$-H\"older space of $\mathbb B$-valued functions equipped with the semi-norm and norm:
\begin{align*}
\|\gamma\|_{t_0,t_1;\alpha}&=  \sup_{t_0\leq t<s\leq t_1} \frac{\|\gamma(s)-\gamma(t)\|_{\mathbb B}}{|s-t|^{\alpha}}, \quad \text{for }\gamma\in C^{\alpha}([t_0,t_1];\mathbb B),\\
\|\gamma\|_{C^{\alpha}([t_0,t_1])} &= \max_{t\in[t_0,t_1]} \|\gamma(t)\|_{\mathbb B} +\|\gamma\|_{t_0,t_1;\alpha},\quad \text{for }\gamma\in C^{\alpha}([t_0,t_1];\mathbb B).
\end{align*}
For each $k\in\bN^+$, $\alpha\in (0,1)$, $0\leq t\leq s\leq T$ and $\xi\in \Lambda_t$, define
\begin{align*}
\Lambda^{0,k,\alpha;\xi}_{t,s}=\bigg\{x\in \Lambda_s: \,& x(\tau)= 1_{[0,t]}(\tau)(\tau)  \xi(\tau\wedge t)+1_{(t,s]}(\tau) g(\tau), \,\,\tau\in[0,s],  \\
 &\text{ for some }g\in C^{\alpha}([t,s];\bR^d),\text{ with }g(t)=\xi(t),\,\,\|g\|_{t,s;\alpha}\leq k  \bigg\},
\end{align*}
and furthermore, we set $\Lambda^{0,k,\alpha}_{0,t}=\cup_{\xi\in\bR^d} \Lambda^{0,k,\alpha;\xi}_{0,t}$ for each $t\in [0,T]$. Then Arzel$\grave{\text{a}}$-Ascoli theorem indicates that each $\Lambda^{0,k,\alpha;\xi}_{t,s}$ is compact in $\Lambda_s$. Moreover, it is obvious that $\cup_{k\in\bN^+} \Lambda^{0,k,\alpha}_{0,s}$ is dense in $\Lambda_s^0$ for any $\alpha\in(0,1)$. In addition, by saying $(s,x)\rightarrow (t^+,\xi)$ for some $(t,\xi)\in [0,T)\times \Lambda_t$ we mean $(s,x)\rightarrow (t^+,\xi)$  with $s\in (t,T]$ and $x\in \cup_{k\in\bN^+} \Lambda_{t,s}^{0,k,\alpha;\xi}$ for some $\alpha \in (0,1)$.

We expect the test function space $ \mathscr C_{\sF}^2$ to include the classical solutions. Nevertheless, it is typical that the classical solutions may not be  differentiable in the time variable $t$ and $(\mathfrak{d}_tu,\mathfrak{d}_{\omega}u)$ may not be time-continuous but just measurable in $t$; see \cite{DuQiuTang10,Tang-Wei-2013} for the \text{state}-dependent cases, or one may even refer to the BSDEs that may be thought of as the trivial \text{stochastic} path-independent PDEs.  

\begin{defn}\label{classical-soltn}
We say that $u\in\mathscr C^2_{\sF}$ is a classical supersolution (resp. subsolution) of SPHJB equation \eqref{SHJB}, if     
$u(T,x)\geq (\text{resp. }\leq) G(x)$ for all $x\in\Lambda_T$ a.s. and  for  each $t\in[0,T)$ with $y\in\Lambda_t$,
{\small
\begin{align}
\text{ess}\liminf_{(s,x)\rightarrow (t^+,y)}
	E_{\sF_{t}} \left\{  -\mathfrak{d}_{s}u(s,x)-\bH(s,x,\nabla u(s,x),\nabla^2 u(s,x), \mathfrak{d}_{\omega}\nabla u(s,x)) \right\} \geq 0, \text{ a.s.,}
\\
\text{(resp. }
\text{ess}\limsup_{(s,x)\rightarrow (t^+,y)}
	E_{\sF_{t}} \left\{  -\mathfrak{d}_{s} u(s,x)-\bH(s,x,\nabla u(s,x),\nabla^2 u(s,x), \mathfrak{d}_{\omega}\nabla u(s,x) ) \right\} \leq 0\text{, a.s.).}
\end{align}
}
The function $u$ is a classical solution of SPHJB equation \eqref{SHJB} if it is both a classical subsolution and a classical supersolution.
\end{defn} 
Throughout this paper, we denote by $\overline{\mathscr V}$ the set of all the classical supersolutions of SPHJB equation \eqref{SHJB} and by $\underline{\mathscr V}$ the set of all the classical subsolutions. Set $\overline\phi(t,x)=Le^{L(T-t)}$, and $\underline\phi(t,x)=-Le^{-L(T-t)}$. Straightforward computations indicate that $\overline\phi\in \overline{\mathscr V}$ and $\underline\phi\in \underline{\mathscr V}$  under Assumption $(\cA 1)$. Therefore, we have the following assertion.
\begin{lem}\label{lem-classical-soltn}
Let Assumption $(\cA 1)$ hold. Neither $\overline{\mathscr V}$ nor $\underline{\mathscr V}$ is empty.
\end{lem}

We now introduce the notion of viscosity solutions. For each $(u,\tau)\in \cS^{2}(\Lambda^0;\bR)\times \mathcal T^0$, $\Omega_{\tau}\in\sF_{\tau}$ with $\mathbb P(\Omega_{\tau})>0$ and $\xi\in L^0(\Omega_{\tau},\sF_{\tau};\Lambda^0_{\tau})$\footnote{Each $\xi\in L^0(\Omega_{\tau},\sF_{\tau};\Lambda^0_{\tau})$ is thought of as  $\xi\in L^0(\Omega_{\tau},\sF_{\tau};\Lambda^0)$ satisfying $\xi(\omega)\in \Lambda_{\tau(\omega)}$ for almost all $\omega\in\Omega_{\tau}$.}, we define for each $(k,\alpha)\in \bN^+\times (0,1)$,
{\small
\begin{align*}
\underline{\mathcal{G}}u(\tau,\xi;\Omega_{\tau},k,\alpha):=\bigg\{
\phi\in\mathscr C^2_{\sF}:&
  \text{ there exists } \hat\tau_k \in  \mathcal T^{\tau}_+\text{ such that}\\
&(\phi-u)(\tau,\xi)1_{\Omega_{\tau}}=0=\essinf_{\bar\tau\in\mathcal T^{\tau}} E_{\sF_{\tau}}\left[\inf_{y\in \Lambda^{0,k,\alpha;\xi}_{\tau,\bar\tau}}
(\phi-u)(\bar\tau\wedge \hat{\tau}_k,y)
\right]1_{\Omega_{\tau}}  \text{ a.s.}
\bigg\},\\
\overline{\mathcal{G}}u(\tau,\xi;\Omega_{\tau},k,\alpha):=\bigg\{
\phi\in\mathscr C^2_{\sF}:
&
  \text{ there exists } \hat\tau_k \in  \mathcal T^{\tau}_+\text{ such that}\\
&(\phi-u)(\tau,\xi)1_{\Omega_{\tau}}=0=\esssup_{\bar\tau\in\mathcal T^{\tau}} E_{\sF_{\tau}}\left[\sup_{y\in \Lambda^{0,k,\alpha;\xi}_{\tau,\bar\tau} }
(\phi-u)(\bar\tau\wedge \hat{\tau}_k,y)
\right]1_{\Omega_{\tau}}  \text{ a.s.}
\bigg\}.
\end{align*}
}
Obviously, if $\underline{\mathcal{G}}u(\tau,\xi;\Omega_{\tau},k,\alpha)$ or $\overline{\mathcal{G}}u(\tau,\xi;\Omega_{\tau},k,\alpha)$ is nonempty, there holds $0\leq\tau <T$ on $\Omega_{\tau}$.  


\begin{defn}\label{defn-viscosity}
We say that $u\in \cS^2 (\Lambda^0;\bR)$ is a viscosity subsolution (resp. supersolution) of SPHJB equation \eqref{SHJB}, if $u(T,x)\leq (\text{ resp. }\geq) G(x)$ for all $x\in\Lambda^0_T$ a.s., and for each $(K_0,\alpha_0)\in \bN^+\times (0,1)$, there exists $(k,\alpha)\in\bN^+\times (0,1)$ with $k\geq K_0$ and $\alpha\leq \alpha_0$, such that for any $\tau\in  \mathcal T^0$, $\Omega_{\tau}\in\sF_{\tau}$ with $\mathbb P(\Omega_{\tau})>0$ and $\xi\in L^0(\Omega_{\tau},\sF_{\tau};\Lambda^0_{\tau})$ and any $\phi\in \underline{\cG}u(\tau,\xi;\Omega_{\tau},k,\alpha)$ (resp. $\phi\in \overline{\cG}u(\tau,\xi;\Omega_{\tau},k,\alpha)$), there holds
{\small
\begin{align}
&\text{ess}\liminf_{(s,x)\rightarrow (\tau^+,\xi)}
	E_{\sF_{\tau}} \left\{ -\mathfrak{d}_{s}\phi(s,x)-\bH(s,x,\nabla \phi(s,x),\nabla^2\phi(s,x),\mathfrak{d}_{\omega}\nabla\phi(s,x)) \right\}  \leq\ \,0, 
\label{defn-vis-sub}
\end{align}
}
 for almost all $ \omega\in\Omega_{\tau}$
 {\small
\begin{align}
\text{(resp.} \quad &\text{ess}\!\!\limsup_{(s,x)\rightarrow (\tau^+,\xi)} 
	\!\!E_{\sF_{\tau}} 
		\left\{ -\mathfrak{d}_{s}\phi(s,x)-\bH(s,x,\nabla \phi(s,x),\nabla^2\phi(s,x),\mathfrak{d}_{\omega}\nabla\phi(s,x)) \right\}  \geq\ \,0,   
		\label{defn-vis-sup}
\end{align}
}
 for almost all $ \omega\in\Omega_{\tau}$).
 
The function $u$ is a viscosity solution of SPHJB equation \eqref{SHJB} if it is both a viscosity subsolution and a viscosity supersolution.
\end{defn}

To make sense of the involved vertical derivatives, a classical (semi)solution is defined on path space $\Lambda$, while the viscosity solution is just defined on $\Lambda^0$. 
Throughout this paper, we define for each $\phi\in\mathscr C^2_{\sF}$, $v\in U$, $t\in [0,T]$, and $x_t\in \Lambda_t$,
\begin{align*}
\mathscr L^{v}\phi(t,x_t)=\,
&
 \mathfrak{d}_t \phi (t,x_t)      +\beta'(t,x_t,v)\nabla\phi(t,x_t) 
 \\
 &+  \text{tr}\left\{\frac{1}{2}
 \sigma(t,x_t,v)\sigma'(t,x_t,v)\nabla ^2\phi(t,x_t) + \sigma(t,x_t,v)\mathfrak{d}_{\omega} \nabla \phi(t,x_t)
 \right\} .
\end{align*}

\begin{rmk}\label{rmk-bH}
In view of the assumption $ (\cA 1)$, for each $\phi\in \mathscr C_{\sF}^2$,  there exists a finite partition $0=\underline t_0<\underline t_1<\ldots<\underline t_n=T$,  such that  for any $0<\delta<\max_{0\leq j \leq n-1} |\underline t_{j+1}-\underline t_j|$,  there exist an $\sF_t$-adapted process $\zeta^{\phi}$ and a constant $L^{\phi}_{\alpha} \in (0,\infty) $ satisfying that a.s. for all $t\in \cup_{0\leq j \leq n-1} (\underline t_j, \underline t_{j+1}-\delta]$ and all $x_t,\bar x_t\in \Lambda_t$,  
 we have
\begin{align*}
 &\Big|  -\mathfrak{d}_{t}\phi(t,x_t)-\bH(t,x_t,\nabla \phi(t,x_t),\nabla^2\phi(t,x_t),\mathfrak{d}_{\omega}\nabla\phi(t,x_t)) 
  \Big| 
   \leq 
  \sup_{v\in U}
  \Big|  \mathscr L^{v}\phi(t,x_t) + f(t,x_t,v) \Big|
  \leq \zeta^{\phi}_t,
\end{align*}
and
\begin{align}
 & \Big| 
  \left\{ -\mathfrak{d}_{t}\phi-\bH(\nabla \phi,\nabla^2\phi, \mathfrak{d}_{\omega}\nabla\phi ) \right\}(t,x_t)
  -\left\{ -\mathfrak{d}_{t}\phi-\bH(\nabla \phi,\nabla^2\phi, \mathfrak{d}_{\omega}\nabla\phi ) \right\}(t,\bar x_t)\Big|\nonumber\\
  &
  \leq 
  \sup_{v\in U}
  \Big| 
   \left( \mathscr L^{v}\phi(t,x_t) +f(t,x_t,v) \right)
  -\left( \mathscr L^{v}\phi(t,\bar x_t)  + f(t,\bar x_t,v)  \right)\Big|
  \nonumber
  \\
  &
  \leq L^{\phi}_{\alpha}  \left(\|x_t-\bar x_t\|^{\alpha}_0 +\|x_t-\bar x_t\| _0  \right),
 \label{R-Lip-const}
\end{align}
where $\zeta^{\phi} \in L^2(\Omega\times [\underline t_j, \underline t_{j+1}-\delta])$ for $j=0,\ldots,n-1$, and  $\alpha$ is the exponent associated to $\phi \in \mathscr C^2_{\sF}$.
Therefore, the conditional expectations in \eqref{defn-vis-sub} and \eqref{defn-vis-sup} are well-defined a.e..
\end{rmk}

\section{Generalized It\^o-Kunita-Wentzell-Krylov formula and its applications to semisolutions}
First, under assumption $(\cA1)$, the following assertions may be obtained via standard computations; refer to \cite{da2014stochastic,mohammed1984stochastic,yong-zhou1999stochastic} for instance.
\begin{lem}\label{lem-SDE}
Let $(\cA1)$ hold. Given $\theta\in\cU$, for the strong solution of SDE \eqref{state-proces-contrl},  for all $p>0$, there exists $K>0$  such that, for all $0\leq r \leq t\leq s \leq T$,  and $\xi\in L^0(\Omega,\sF_r;\Lambda_r)$,
 \\[3pt]
(i)   the two processes $\left(X_s^{r,\xi;\theta}\right)_{t\leq s \leq T}$ and $\left(X^{t,X_t^{r,\xi;\theta};\theta}_s\right)_{t\leq s\leq T}$ are indistinguishable;\\[2pt]
(ii)  $E_{\sF_r} \left[ \max_{r\leq l \leq T} \left\|X^{r,\xi;\theta}_l\right\|_0^p\right] \leq  K \left(1+ \|\xi\|_0^p\right)$ a.s.;\footnote{Here, denoting by $x_r$ a path in $\Lambda_r$, we set $E_{\sF_r} \left[ \max_{r\leq l \leq T} \left\|X^{r,\xi;\theta}_l\right\|_0^p\right]= E_{\sF_r} \left[ \max_{r\leq l \leq T} \left\|X^{r,x_r;\theta}_l\right\|_0^p\right] \Big|_{x_r=\xi} $; the conditional expectation in assertion (iv) is defined analogously. }
\\[2pt]
(iii) $ E_{\sF_r} \left[  \left| d_0( X^{r,\xi;\theta}_s,\,X^{r,\xi;\theta}_t  ) \right|^p \right]  
				\leq  K \left(|s-t|^p+  |s-t|^{p/2} \right)$ a.s.;\\[2pt]
(iv) given another $\hat{\xi}\in   L^0(\Omega,\sF_r;\Lambda_r)$, 
 $E_{\sF_r} \left[ \max_{r\leq l \leq T} \left\|X^{r,\hat\xi;\theta}_l-X^{r,\xi;\theta}_l\right\|_0 ^{p+1} \right]
			\leq  K  \|\xi-\hat\xi\|_0^{p+1}, \text{ a.s.;}
$\\
(v) the constant $ K$ depends only on $L,\,p,$ and $T$.
\end{lem}

To investigate the H\"older continuity of the paths, we recall a general version of Kolmogorov criterion by Revuz and Yor \cite[Theorem (2.1), Page 26--28]{Revuz_Yao_2013continuous}.
\begin{lem}\label{lem-alpha}
Given a Banach space $(\mathbb B, \,\|\cdot\|_{\mathbb B})$, let $(Y(t))_{t\in[0,T]}$ be a $\mathbb B$-valued stochastic process for which there are three strictly positive constants $q,\lambda$,  and $\delta$ such that 
$$
E\left[  \left\|  Y(t)-Y(s)   \right\|_{\mathbb B}^{q}\right] \leq \lambda |t-s|^{1+\delta}, \quad \text{for all }0\leq t\leq s\leq T.
$$
Then, for each $\alpha\in \left(0,\frac{\delta}{q} \right)$ the process $Y$ admits an $\alpha$-H\"older continuous modification (denoted by itself) such that 
\begin{align}
E\left[
\left\|  Y \right\|_{0,T;\alpha} ^{q} 
\right] 
\leq C, \label{holder-est}
\end{align}
where the constant $C$ depends on $\lambda,\alpha,q,\delta$, and $T$.
\end{lem}
\begin{rmk}\label{rmk-holder-contin}
The controlled SDE in Lemma \ref{lem-SDE} may be considered in any finite interval $[0,N]$ for $N>0$ and the time $T$ may also be general $T>0$.  By assertion (iii) of Lemma \ref{lem-SDE}, the arbitrariness of $p$ therein and Lemma \ref{lem-alpha} imply that for all $\alpha\in(0,\frac{1}{2})$, $\tau >0$, and $q>1$, there exists constant $C>0$ such that for all  $\theta\in \cU$, $\xi\in L^0(\Omega,\sF_r;\Lambda_r)$, and $r\geq 0$,  there holds  $E_{\sF_r}\left[\left\|  X^{r,\xi;\theta} \right\|_{r,r+\tau;\alpha} ^{q} 
\right]<C$ a.s.,  with $C$ depending only on $L, \tau,q$, and $\alpha$. Further, for each $\theta\in\cU$, recalling that for each $0<\alpha<\alpha'<\frac{1}{2}$, and $0\leq r<t\leq T$,  there holds
$$
\left\|  X^{r,\xi;\theta} \right\|_{r,t;\alpha} \leq \left\|  X^{r,\xi;\theta} \right\|_{r,T;\alpha'} |t-r|^{\alpha'-\alpha}, \quad \text{a.s.,}
$$
we have the stopping times 
\begin{align*}
\tau^{\theta}_{k,\alpha}:=\inf\{s> r; \|X^{r,\xi;\theta}\|_{r,s;\alpha} >k   \} \wedge T, \text{ for }k>0,
\end{align*}
 well-defined, with $\bP (r<\tau^{\theta}_{k,\alpha})=1$ and  $ \tau^{\theta}_{k,\alpha} $ increasingly converging to $T$ as $k\rightarrow \infty$.
\end{rmk}

We then generalize an It\^o-Kunita-Wentzell-Krylov formula (see \cite[Pages 118-119]{kunita1981some} for instance) for the composition of random fields and stochastic differential equations to our \textit{path-dependent} setting. Recall that for each $\phi\in\mathscr C^2_{\sF}$, $v\in U$, $t\in [0,T]$, and $x_t\in \Lambda_t$,
{\small
\begin{align*}
\mathscr L^{v}\phi(t,x_t)=\,
&
 \mathfrak{d}_t \phi (t,x_t)      +\beta'(t,x_t,v)\nabla\phi(t,x_t) 
 \\
 &+  \text{tr}\left\{\frac{1}{2}
 \sigma(t,x_t,v)\sigma'(t,x_t,v)\nabla ^2\phi(t,x_t) + \sigma(t,x_t,v)\mathfrak{d}_{\omega} \nabla \phi(t,x_t)
 \right\} .
\end{align*}
}
\begin{lem}\label{lem-ito-wentzell}
  Let assumption $(\cA1)$ hold.
 Suppose  $u\in\mathscr C_{\sF}^2$ with the associated partition $0=\underline t_0<\underline t_1<\ldots<\underline t_n=T$. Then, for each $\theta\in\cU$, it holds almost surely that, for each $\underline t_j \leq \varrho \leq \tau < \underline t_{j+1}$, $j=0,\ldots,n-1$, and $x_{\varrho}\in \Lambda_{\varrho}$, it holds that
    \begin{align}
  u( {\tau},X^{{\varrho},x_{\varrho};\theta}_ {\tau}) &= u(\underline t_j,x_{{\varrho}}) +
     	\!\int_{\varrho}^{\tau}    \mathscr L^{\theta(s)} u\left(s,X^{{\varrho},x_{\varrho};\theta}_s \right)    \,ds     \nonumber \\
	&\quad +\int_{\varrho}^{\tau} \left( (\nabla u)'(r,X^{{\varrho},x_{\varrho};\theta}_r)  \sigma(r,X^{{\varrho},x_{\varrho};\theta}_r,\theta(r)) +  \mathfrak{d}_{\omega}u(r,X^{{\varrho},x_{\varrho};\theta}_r)    
     \right)\,dW(r),\text{ a.s..} \label{eq-ito}
   \end{align}
\end{lem}

\begin{proof}
W.l.o.g., we only prove \eqref{eq-ito} for $ {\tau} \in (0,\underline t_1)$, ${\varrho}=0$ and $x_0=x\in \bR^d$. 
For each $N\in\bN^+$ with $N>2$, letting $t_i=\frac{i {\tau}}{N}$ for $i=0,1,\dots,N$, we get a partition of $[0, {\tau}]$ with $0=t_0<t_1<\cdots<t_{N-1}<t_N= {\tau}$. For each $\theta\in\cU$, set
$$
^N\!X(t)=\sum_{i=0}^{N-1} X^{0,x;\theta}({t_i})1_{[t_i,t_{i+1})(t)} + X^{0,x;\theta}( {\tau})1_{\{ {\tau}\}}(t), \quad \text{for }t\in[0, {\tau}],
$$
and $^N\!X_{t-}(s)= \, ^N\!X(s)1_{[0,t)}(s) + \lim_{r\rightarrow t^-} \, ^N\!X(r) 1_{\{t\}}(s)$, for $0\leq s\leq t\leq  {\tau}$. Due to the time-continuity of $X^{0,x;\theta}$, there holds the following approximation:
$$
\lim_{N\rightarrow \infty} \|X^{0,x;\theta}-^N\!X\|_0+\|X^{0,x;\theta}_t-^N\!X_{t-}\|_0=0\quad \text{for all }t\in(0, {\tau}],\quad \text{a.s..}
$$

Then, we have
\begin{align}
u( {\tau},^N\!X_ {\tau})-u(0,x)
&=\!\! \sum_{i=0}^{N-1} u(t_{i+1},^N\!X_{t_{i+1}-})-u(t_i,^N\!X_{t_i}) 
+\!\! \sum_{i=0}^{N-1} u(t_{i+1},^N\!X_{t_{i+1}})-u(t_{i+1},^N\!X_{t_{i+1}-})
\nonumber\\
:&= I^{(N)}_1+I^{(N)}_2. \label{eq-docomp}
\end{align}
%
%
As $u\in\mathscr C_{\sF}^2$, it holds that
{\small
\begin{align*}
I^{(N)}_1
&=\sum_{i=0}^{N-1}u(t_{i+1},^N\!X_{t_{i+1}-})-u(t_i,^N\!X_{t_i})
=\sum_{i=0}^{N-1}\int_{t_i}^{t_{i+1}}  \mathfrak{d}_tu(s,^N\!X_{s-})\,ds 
	+ \int_{t_i}^{t_{i+1}} \mathfrak{d}_{\omega}u(r,^N\!X_{r-})    
     \,dW(r),
\end{align*}
}
which, as $N$ tends to infinity, converges in probability to 
\begin{align}
\int_0^ {\tau}\mathfrak{d}_ru(r,X^{0,x;\theta}_r)\,dr
+\int_0^ {\tau} \mathfrak{d}_{\omega} u (r,X^{0,x;\theta}_r) \,dW(r). \label{ito-term-1}
\end{align}

On the other hand, by the definition of vertical derivatives, it holds that
\begin{align*}
I^{(N)}_2
&=\sum_{i=0}^{N-1}  \left( \nabla u(t_{i+1},^N\!X_{t_{i+1}-}) \right)' \left( ^N\!X(t_{i+1})-    ^N\!X({t_{i+1}-}) \right)
\\
&\quad 
+ \frac{1}{2}\sum_{i=0}^{N-1}  \left( ^N\!X(t_{i+1})-    ^N\!X({t_{i+1}-})   \right) ' \nabla^2u (t_{i+1},^N\!X_{t_{i+1}-} ^h) \left( ^N\!X(t_{i+1})-    ^N\!X({t_{i+1}-})   \right) 
\\
&:=M_1^{(N)} + M_2^{(N)},
\end{align*}
for some $h$ satisfying $\left|h-^N\!X({t_{i+1}-}) \right| \leq \left|  ^N\!X(t_{i+1})-    ^N\!X({t_{i+1}-})  \right|$. Further, we have
{\small
\begin{align*}
M_1^{(N)}
&
	=
	\left( \nabla u(t_1,^N\! X_{t_1-}) \right)' (^N\! X(t_1)-x) 
	+
	\sum_{i=1}^{N-1} \left( \nabla u(t_{i},^N\!X_{t_{i}}) \right)' \left( ^N\!X(t_{i+1})-    ^N\!X({t_{i+1}-}) \right)
\\
&\quad
	+\sum_{i=1}^{N-1} \left(\nabla u(t_{i+1},^N\!X_{t_{i+1}-}) - \nabla u(t_{i},^N\!X_{t_{i}}) \right) \left( ^N\!X(t_{i+1})-    ^N\!X({t_{i+1}-}) \right).
\end{align*}
}
Notice that
{\small
\begin{align*}
&
\sum_{i=1}^{N-1} \left(\nabla u(t_{i+1},^N\!X_{t_{i+1}-}) - \nabla u(t_{i},^N\!X_{t_{i}}) \right)' \left( ^N\!X(t_{i+1})-    ^N\!X({t_{i+1}-}) \right)
\\
&
=
\sum_{i=1}^{N-1} \left(\nabla u(t_{i+1},^N\!X_{t_{i+1}-}) - \nabla u(t_{i},^N\!X_{t_{i}}) \right)'  \int^{t_{i+1}}_{t_{i}} \!\!\!\beta(r,X^{0,x ;\theta}_r,\theta(r)) dr  
\\
&\quad
+\sum_{i=1}^{N-1} \left(\nabla u(t_{i+1},^N\!X_{t_{i+1}-}) - \nabla u(t_{i},^N\!X_{t_{i}}) \right)'
\int^{t_{i+1}}_{t_{i}} \!\!\! \sigma(r,X^{0,x ;\theta}_r,\theta(r))\,dW(r) 
\\
&:=P_1^{(N)} + P_2^{(N)},
\end{align*}
}
where $P_1^{(N)}$ converges to  zero in probability  due to the boundedness of $\beta$ and the continuity of $\nabla u$ and in view of (i)-(b) for $\mathfrak{d}_{\omega}\nabla u$ in Definition \ref{defn-testfunc}, we have  $P_2^{(N)}$ converge in probability to 
$\int_0^ {\tau}    \text{tr}\left\{ \sigma(t,X^{0,x;\theta}_t,\theta(t)) \mathfrak{d}_{\omega} \nabla u (t,X^{0,x;\theta}_t)\right\}  dt$. This combined with some standard computations yields the convergence (in probability) of $M_1^{(N)}$ with the limit being
\begin{align}
&\int_0^ {\tau}  \left( \nabla u(t,X_t^{0,x;\theta}) \right)' \,dX^{0,x;\theta}(t)
+ \int_0^ {\tau} \sum_{i=1}^d\sum_{l=1}^{m}    
  \sigma^{il}(t,X_t^{0,x;\theta},\theta(t)) (\mathfrak{d}_{\omega} \nabla u)^{li}(t,X_t^{0,x;\theta}) \, dt
\nonumber\\
&=
	\int_0^ {\tau} \left(  \beta'(t,X^{0,x;\theta}_t,\theta(t)) \nabla u(t,X_t^{0,x;\theta}) +  \text{tr}\left\{ \sigma(t,X^{0,x;\theta}_t,\theta(t)) \mathfrak{d}_{\omega} \nabla u (t,X^{0,x;\theta}_t)\right\} \right) dt
	\nonumber\\
&\quad	
+\int_0^ {\tau}  (\nabla u)'(r,X^{0,x;\theta}_r) \sigma(r,X^{0,x;\theta}_r,\theta(r)) \, dW(r). \label{ito-term-2}
\end{align}
Meanwhile, straightforward standard calculations give the convergence of $M_2^{(N)}$  to
\begin{align}
&\frac{1}{2}\int_0^ {\tau} \sum_{i,j=1}^d    
 \nabla^2_{ij} u(t,X_t^{0,x;\theta}) \,d \left\langle X^{0,x;\theta} \right\rangle^{ij}(t)
\nonumber \\
&
=
	\int_0^ {\tau}  \frac{1}{2} \text{tr}\left\{ \sigma(t,X^{0,x;\theta}_t,\theta(t))\sigma'(t,X^{0,x;\theta}_t,\theta(t)) \nabla^2 u (t,X^{0,x;\theta}_t)\right\} dt. \label{ito-term-3}
\end{align}

 In the course of approaching the limits \eqref{ito-term-1}, \eqref{ito-term-2}, and \eqref{ito-term-3}, the dominated convergence and the dominated convergence theorem for stochastic integrals (\cite[Chapter IV, Theorem 32]{pP05})  imply that the Lebesgue integrals converge almost surely and the stochastic integrals in probability. Finally, summing up all the obtained convergences yields the desired equality.
\end{proof}

Now, we discuss some properties of classical/viscosity semisolutions. 
\begin{thm}\label{thm-classical-viscosity}
Let Assumption $(\cA 1)$ hold. Each classical subsolution (resp. supersolution) is a viscosity subsolution (resp. supersolution), and thus, each classical solution is a viscosity solution.
\end{thm}
\begin{proof}
\textbf{Step 1.} We first prove that each $ \overline \mu \in \overline{\mathscr V} $ is a viscosity supersolution. Indeed, for each $\phi\in \overline\cG \overline \mu(\tau,\xi_{\tau};\Omega_{\tau},k,\alpha)$ with $k>0$ and $\alpha\in (0,\frac{1}{2})$, $\tau\in \mathcal T^0$, $\Omega_{\tau}\in\sF_{\tau}$, $\mathbb P(\Omega_{\tau})>0$, and $\xi_{\tau}\in L^0(\Omega_{\tau},\sF_{\tau};\Lambda_{\tau}^0)$, the relation \eqref{defn-vis-sup} holds  for almost all $ \omega\in\Omega_{\tau}$. Suppose that, to the contrary, there exist $\eps, \tilde\delta>0$ and $\Omega'\in\sF_{\tau}$ with $\Omega'\subset \Omega_{\tau}$, $\bP(\Omega')>0$, such that a.e. on $\Omega'$,
\begin{align*}
\esssup_{
s \in (\tau,(\tau+4\tilde\delta^2) \wedge T],\,x\in B_{2\tilde\delta}(\xi_{\tau}) \cap \Lambda^{0,k,\alpha;\xi_{\tau}}_{\tau,s\wedge T}	
			} 
	\!\!\! E_{\sF_{\tau}}
		 \left\{ -\mathfrak{d}_{s}\phi(s,x)-\bH(s,x,\nabla \phi(s,x),\nabla^2\phi(s,x),\mathfrak{d}_{\omega}\nabla\phi(s,x)) \right\} 
		 \leq -\eps.
\end{align*}

Let  $\hat{\tau}_k$ be the stopping time associated to the fact $\phi\in \overline\cG \overline \mu (\tau,\xi_{\tau};\Omega_{\tau},k)$. We may think of $\xi$ valued in $\Lambda^0_T$, with $\xi(t)=\xi_{ \tau}(t\wedge \tau)$ for all $t\in [0,T]$.
Notice that, associated to $\phi\in\mathscr C_{\sF}^2$ and $\overline \mu \in\mathscr C_{\sF}^2$, the two partitions may be combined into one: $0=\underline t_0<\underline t_1<\ldots<\underline t_n=T$. W.l.o.g., we assume $\tilde{\delta}\in (0,1)$, and $\Omega'=\{[\tau,\tau+4\tilde{\delta}^2] \subset [\underline t_j,\underline t_{j+1})\}=\Omega$ for some $j\in\{0,1,\ldots, n-1\}$.


For each $\theta\in\cU$, define
$\tau^{\theta}=\inf\{ s>\tau: X^{\tau,\xi_{\tau};{\theta}}_s \notin B_{\tilde\delta}(\xi_{\tau}) \} \wedge T$, and set
\begin{align*}
\tau^{\theta}_{k,\alpha} =\inf\{s> \tau; \|X^{\tau,\xi_{\tau};\theta}\|_{\tau,s;\alpha} >k   \} \wedge T. 
\end{align*} 
Then $\bP(\tau^{\theta}  >\tau,\,\tau^{\theta}_{k,\alpha} >\tau)=1$. Letting $\bar \tau= \hat\tau_k \wedge \tau^{\theta} \wedge \tau^{\theta}_{k,\alpha} \wedge (\tau+\tilde \delta^2) \wedge T$, we have $\bP(\bar\tau >\tau)=1$.  Then, for each $\theta\in\cU$,
\begin{align}
\essinf_{s\in (\tau,\bar\tau]}  E_{\sF_{\tau}} \left[ 
 \mathscr L^{\theta(s)} \phi (s,X_s^{\tau,\xi_{\tau};\theta})  +f(s,X_s^{\tau,\xi_{\tau};\theta},\theta(s)) \right] \geq \eps.
 \label{relation-1}
\end{align}

On the other hand, as $ \overline \mu \in \overline{\mathscr V} $,  it holds that for  all $t\in[0,T)$ with $y\in\Lambda_t^0$,
\begin{align*}
\text{ess}\liminf_{(s,x)\rightarrow (t^+,y)}
	E_{\sF_{t}} \left\{  -\mathfrak{d}_{s}\overline \mu(s,x)-\bH(s,x,\nabla \overline \mu(s,x),\nabla^2 \overline \mu(s,x), \mathfrak{d}_{\omega}\nabla \overline \mu(s,x)) \right\} \geq 0, \text{ a.s.,}
\end{align*}
which implies that there exists $\delta \in (0,\tilde \delta)$ such that 
\begin{align*}
\essinf_{
s \in (\tau,(\tau+4\delta^2) \wedge T]	
			} 
	E_{\sF_{\tau}}
		 \left\{ -\mathfrak{d}_{s}\overline \mu(s,\xi_s)-\bH(s,\xi_s,\nabla \overline \mu(s,\xi_s),\nabla^2\overline \mu(s,\xi_s),\mathfrak{d}_{\omega}\nabla\overline \mu(s,\xi_s)) \right\} 
		 \geq -\frac{\eps}{4}.
\end{align*}
Take $\tilde \tau =\bar \tau \wedge (\tau+\delta^2) \wedge T$.  By the measurable selection theorem, there exists $\tilde \theta\in\cU$ such that 
\begin{align*}
\esssup_{
s \in (\tau,\tilde \tau]	
			} 
	E_{\sF_{\tau}}
		 \left\{ \mathscr L^{\tilde\theta(s)} \overline \mu(s,\xi_s)
		 +f(s,\xi_s,\tilde\theta(s)) \right\} 
		 \leq \frac{\eps}{2},
\end{align*}
which together with Remark \ref{rmk-bH} indicates that 
\begin{align}
  \int_{\tau}^{ \tilde \tau} 
  E_{\sF_{\tau}} \left[ 
 \mathscr L^{\tilde \theta(s)} \overline \mu (s,X_s^{\tau,\xi_{\tau};\tilde \theta})  +f(s,X_s^{\tau,\xi_{\tau};\tilde\theta},\tilde\theta(s)) \right] ds
 \leq \int_{\tau}^{ \tilde \tau} E_{\sF_{\tau}} \left[
  \frac{\eps}{2} + L^{\overline \mu}_{\tilde \alpha} |\tilde\delta^{\tilde\alpha}+\tilde \delta |  \right] ds \quad \text{a.s.,}
 \label{relation-2}
\end{align}
where $\tilde \alpha$ is the exponent associated to $\overline \mu\in \mathscr C^2_{\sF}$. Combining \eqref{relation-1} and \eqref{relation-2} gives
\begin{align*}
  \int_{\tau}^{ \tilde \tau}    E_{\sF_{\tau}} \left[ 
 \mathscr L^{\tilde \theta(s)} (\phi-\overline \mu) (s,X_s^{\tau,\xi_{\tau};\tilde \theta})    \right] \,ds\geq  \int_{\tau}^{ \tilde \tau} E_{\sF_{\tau}} \left[ 
  \frac{\eps}{2} -  L^{\overline \mu}_{\tilde \alpha}  |\tilde\delta^{\tilde\alpha}+\tilde \delta |  \right] ds,\quad \text{a.s..}
\end{align*}
Applying the It\^o-Kunita-Wentzell-Krylov formula in Lemma \ref{lem-ito-wentzell} further yields that 
\begin{align*}
E \left[ (\phi -\overline \mu)(\bar\tau, \,X_{\bar \tau}^{\tau,\xi_{\tau};\tilde \theta})  \right]
&= E  \left[ (\phi -\overline \mu)(\tau, \,\xi_{\tau}) 
+\int_{\tau}^{\tilde\tau} 
 \mathscr L^{\tilde \theta(s)} (\phi-\overline \mu) (s,X_s^{\tau,\xi_{\tau};\tilde \theta})\,ds
 \right]
 \\
 &
 =E \left[  
\int_{\tau}^{\tilde\tau} 
 \mathscr L^{\tilde \theta(s)} (\phi-\overline \mu) (s,X_s^{\tau,\xi_{\tau};\tilde \theta})\,ds
 \right]\\
 &\geq E \left[ \int_{\tau}^{ \tilde \tau} \left(
  \frac{\eps}{2} -  L^{\overline \mu}_{\tilde \alpha}  |\tilde\delta^{\tilde \alpha}+\tilde \delta |  \right)ds \right],
\end{align*}
which is $>0$ when $\tilde \delta$ is sufficiently small, contradicting with $\phi\in \overline\cG \overline \mu(\tau,\xi_{\tau};\Omega_{\tau},k,\alpha)$.

\textbf{Step 2.} To prove that each $ \underline \mu \in \underline{\mathscr V} $ is a viscosity subsolution, it is sufficient to verify that for each $\phi\in \underline\cG \underline \mu(\tau,\xi_{\tau};\Omega_{\tau},k,\alpha)$ with $k>0$ and $\alpha\in (0,\frac{1}{2})$, $\tau\in \mathcal T^0$, $\Omega_{\tau}\in\sF_{\tau}$, $\mathbb P(\Omega_{\tau})>0$, and $\xi_{\tau}\in L^0(\Omega_{\tau},\sF_{\tau};\Lambda_{\tau}^0)$, the relation \eqref{defn-vis-sub} holds  for almost all $ \omega\in\Omega_{\tau}$. 

To the contrary, suppose that there exist $\bar \eps, \bar \delta\in (0,1)$ and $\Omega'\in\sF_{\tau}$ with $\Omega'\subset \Omega_{\tau}$, $\bP(\Omega')>0$, such that a.e. on $\Omega'$,
\begin{align*}
\essinf_{
s \in (\tau,(\tau+4\bar\delta^2) \wedge T],\,x\in B_{2\bar\delta}(\xi_{\tau}) \cap \Lambda^{0,k,\alpha;\xi_{\tau}}_{\tau,s\wedge T}	
			} 
	E_{\sF_{\tau}}
		 \left\{ -\mathfrak{d}_{s}\phi(s,x)
		 -
		 \bH(s,x,\nabla \phi(s,x),\nabla^2\phi(s,x),\mathfrak{d}_{\omega}\nabla\phi(s,x))
		  \right\} 
		 \geq 2\bar\eps.
\end{align*}
Let  $\hat{\tau}_k$ be the stopping time associated to $\phi\in \underline\cG \underline \mu (\tau,\xi_{\tau};\Omega_{\tau},k,\alpha)$. Again, we think of $\xi$ as  a path in $\Lambda^0_T$, with $\xi(t)=\xi_{ \tau}(t\wedge \tau)$ for all $t\in [0,T]$, and the two partitions,  associated to $\phi\in\mathscr C_{\sF}^2$ and $\underline \mu \in\mathscr C_{\sF}^2$,  are combined into one: $0=\underline t_0<\underline t_1<\ldots<\underline t_n=T$. W.l.o.g., we assume  $\Omega'=\{[\tau,\tau+4\bar{\delta}^2] \subset [\underline t_j,\underline t_{j+1})\}=\Omega$ for some $j\in\{0,1,\ldots, n-1\}$.

By   the measurable selection theorem, there exists $\bar\theta\in \cU$ such that a.s.,
$$
- \mathscr L^{\bar\theta(s)}\phi(s,\xi_s)  -f(s,\xi_s,\bar\theta(s))   \geq   -\mathfrak{d}_{s}\phi(s,\xi_s)-\bH(s,\xi_s,\nabla \phi(s,\xi_s),\nabla^2\phi(s,\xi_s),\mathfrak{d}_{\omega}\nabla\phi(s,\xi_s)) -       \bar\eps,
$$
for almost all $s$ satisfying $\tau < s<(\tau+ 4 \bar \delta^2) \wedge T$. Thus, we have 
\begin{align}
\essinf_{\tau < s<\left(\tau+ 4\bar \delta^2\right)\wedge T} 
	E_{\sF_{\tau}}
		\{ - \mathscr L^{\bar\theta(s)}\phi(s,\xi_s)  -f(s,\xi_s,\bar\theta(s)) \}  \geq \bar\eps,\quad \text{a.e..} \label{est-phi}
\end{align}
Set
$\tau^{\bar\theta}=\inf\{ s>\tau: X^{\tau,\xi_{\tau};{\bar\theta}}_s \notin B_{\bar\delta}(\xi_{\tau}) \} \wedge T$, and
 \begin{align*}
\tau^{\bar \theta}_{k,\alpha}=\inf\{s> \tau; \|X^{\tau,\xi_{\tau};\bar\theta}\|_{\tau,s;\alpha} >k   \} \wedge T. 
\end{align*} 
Then $\bP(\tau^{\bar\theta}  >\tau,\,  \tau^{\bar\theta}_{k,\alpha}>\tau)=1$.   Letting $\bar \tau= \hat\tau_k \wedge \tau^{\bar\theta} \wedge \tau^{\bar\theta}_{k,\alpha} \wedge (\tau+ \bar \delta^2) \wedge T$, we have $\bar\tau >\tau$ a.s.. Combining \eqref{est-phi} and the analysis in Remark \ref{rmk-bH} yields that for all $\tau' \in \mathcal T^{\tau}$,
\begin{align}
\int_{\tau}^{\tau'\wedge \bar\tau}\!\!\!
E_{\sF_{\tau}} \left[ 
 \mathscr L^{\bar \theta(s)} \phi (s,X_s^{\tau,\xi_{\tau};\bar\theta})  +f(s,X_s^{\tau,\xi_{\tau};\bar\theta},\bar\theta(s)) 
 	\right] 
	ds
\leq
	\int_{\tau}^{\tau'\wedge \bar\tau} \!\!\!
	E_{\sF_{\tau}} \left[
	L^{\phi}_{\alpha'} |\bar \delta ^{\alpha'} +\bar \delta|    -\bar \eps \right] ds, \label{relation-3}
\end{align}
where $\alpha'\in (0,1)$ is the exponent associated to $\phi\in\mathscr C^2_{\sF}$.

On the other hand,   as $ \underline \mu \in \underline{\mathscr V} $,  there exists $\delta \in (0,\bar \delta)$ such that 
\begin{align*}
\esssup_{
s \in (\tau,(\tau+4\delta^2) \wedge T]	
			} 
	E_{\sF_{\tau}}
		 \left\{ -\mathfrak{d}_{s}\underline \mu(s,\xi_s)-\bH(s,\xi_s,\nabla \underline \mu(s,\xi_s),\nabla^2\underline \mu(s,\xi_s),\mathfrak{d}_{\omega}\nabla\underline \mu(s,\xi_s)) \right\} 
		 \leq 0,
\end{align*}
which particularly implies that
\begin{align*}
\essinf_{
s \in (\tau,(\tau+4\delta^2) \wedge T]	
			} 
	E_{\sF_{\tau}}
		 \left\{ 
		  \mathscr L^{\bar\theta(s)}\underline \mu (s,\xi_s)  +f(s,\xi_s,\bar\theta(s))
		 \right\} 
		 \geq 0,
\end{align*}
Take $\tilde \tau =\bar \tau \wedge (\tau+\delta^2) \wedge T$. It obviously holds that $\tilde \tau > \tau$ a.s.. Moreover, Remark \ref{rmk-bH} implies that
\begin{align}
  \int_{\tau}^{ \tilde \tau} 
  E_{\sF_{\tau}} \left[ 
 \mathscr L^{\bar \theta(s)} \underline \mu (s,X_s^{\tau,\xi_{\tau};\bar \theta})  +f(s,X_s^{\tau,\xi_{\tau};\bar\theta},\bar\theta(s)) \right] ds
 \geq \int_{\tau}^{ \tilde \tau} E_{\sF_{\tau}} \left[
   - L^{\underline \mu}_{\tilde \alpha} |\bar\delta^{\tilde \alpha}+\bar \delta |  \right] ds \quad \text{a.s.,}
 \label{relation-4}
\end{align}
where $\tilde \alpha$ is the exponent associated to $\underline \mu\in \mathscr C^2_{\sF}$.

Combining \eqref{relation-3} and \eqref{relation-4} and applying  Lemma \ref{lem-ito-wentzell} yield that
\begin{align*}
E \left[ (\phi -\underline \mu)(\bar\tau, \,X_{\bar \tau}^{\tau,\xi_{\tau};\bar \theta})  \right]
&= E \left[ (\phi -\underline \mu)(\tau, \,\xi_{\tau}) 
+\int_{\tau}^{\tilde\tau} 
 \mathscr L^{\bar \theta(s)} (\phi-\underline \mu) (s,X_s^{\tau,\xi_{\tau};\bar \theta})\,ds
 \right]
 \\
 &
 =E \left[  
\int_{\tau}^{\tilde\tau} 
 \mathscr L^{\bar \theta(s)} (\phi-\underline \mu) (s,X_s^{\tau,\xi_{\tau};\bar \theta})\,ds
 \right]\\
 &\leq E \left[ \int_{\tau}^{ \tilde \tau} \left(
    L^{\underline \mu}_{\tilde \alpha} |\bar\delta^{\tilde \alpha}+\bar \delta |  + L^{\phi}_{\alpha'}|\bar\delta^{\alpha'}+\bar\delta | -\bar\eps\right)ds \right],
\end{align*}
which is $<0$ when $\bar \delta$ is sufficiently small, resulting in a contradiction with $\phi\in \underline\cG \underline \mu(\tau,\xi_{\tau};\Omega_{\tau},k,\alpha)$. 
\end{proof}

\section{Existence of the viscosity solution}

%

The following properties of the value function $V$ hold in a similar way to \cite[Proposition 3.3]{qiu2017viscosity}.  
\begin{prop}\label{prop-value-func}
Let $(\cA 1)$ hold.
\\[4pt]
(i) For each  $t\in[0,T]$, $\eps\in (0,\infty)$, and $\xi\in L^0(\Omega,\sF_t;\Lambda_t)$, there exists $\bar{\theta}\in\cU$ such that
$$
E\left[ J(t,\xi;\bar{\theta})-V(t,\xi)\right]<\eps.
$$
(ii) For each $({\theta},x_0)\in\cU\times \bR^d$, $\left\{J(t,X_t^{0,x_0;\theta};{\theta})-V(t,X_t^{0,x_0;\theta})\right\}_{t\in[0,T]}$ is a supermartingale, i.e.,  for any $0\leq t\leq \tilde{t}\leq T$,
\begin{align}
V(t,X_t^{0,x_0;\theta})
\leq E_{\sF_t}V(\tilde{t},X_{\tilde{t}}^{0,x_0;{\theta}}) + E_{\sF_t}\int_t^{\tilde{t}}f(s,X_s^{0,x_0;{\theta}},\theta(s))\,ds,\,\,\,\text{a.s..}\label{eq-vfunc-supM}
\end{align}
(iii) For each $({\theta},x_0)\in\cU\times \bR^d$, $\left\{V(s,X_s^{0,x_0;{\theta}})\right\}_{s\in[0,T]}$ is a continuous process.
\\[3pt]
(iv) 
 There exists $L_V>0$ such that for each $(\theta,t)\in\cU \times[0,T]$,
$$
|V(t,x_t)-V(t,y_t)|+|J(t,x_t;\theta)-J(t,y_t;\theta)|\leq L_V\|x_t-y_t\|_0,\,\,\,\text{a.s.},\quad \forall\,x_t,y_t\in\Lambda_t,
$$
with $L_V$ depending only on $T$ and $L$.
\\[3pt]
(v) With probability 1, $V(t,x)$ and $J(t,x;\theta)$ for each $\theta\in\cU$ are continuous  on $[0,T]\times\Lambda$ and 
$$ \sup_{(t,x)\in[0,T]\times\Lambda}   \max\left\{|V(t,x_t)|,\,|J(t,x_t;\theta)| \right\}       \leq L(T+1) \quad\text{a.s..}$$
\end{prop}

Following is the dynamic programming principle, whose proof is the same to \cite[Theorem 3.4]{qiu2017viscosity}, utilizing the separability of path spaces.
\begin{thm}\label{thm-DPP}
Let assumption $(\cA1)$ hold. For any stopping times $\tau,\hat\tau $ with $\tau\leq \hat\tau \leq T$, and any $ \xi\in L^0(\Omega,\sF_{\tau};\Lambda_{\tau})$,
 we have
\begin{align*}
V(\tau,\xi)=\essinf_{\theta\in\cU} E_{\sF_{\tau}}
\left[
\int_{\tau}^{\hat\tau} f\left(s,X_s^{\tau,\xi;\theta},\theta(s)\right)\,ds + V\left(\hat\tau,X^{\tau,\xi;\theta}_{\hat \tau}\right)
\right] \quad a.s.
\end{align*}
\end{thm}


Then we are ready to give the existence of the viscosity solution.

\begin{thm}\label{thm-existence}
 Let $(\cA1)$ hold. The value function $V$ defined by \eqref{eq-value-func} is a viscosity solution of the SPHJB equation \eqref{SHJB}.
\end{thm}

\begin{proof}
First, we have $V\in\cS^{\infty}(\Lambda;\bR)$ by Proposition \ref{prop-value-func}. The proof is divided into two steps.

\textbf{Step 1.}
To the contrary, suppose that for each $(k,\alpha)\in \bN^+\times (0,\frac{1}{2})$ with $k\geq K_0$ and $\alpha\leq \alpha_0$ for some existing $(K_0,\alpha_0)\in\bN^+\times (0,1)$,  there exists $\phi\in \underline\cG V(\tau,\xi_{\tau};\Omega_{\tau},k,\alpha)$ with $\tau\in \mathcal T^0$, $\Omega_{\tau}\in\sF_{\tau}$, $\mathbb P(\Omega_{\tau})>0$, and $\xi_{\tau}\in L^0(\Omega_{\tau},\sF_{\tau};\Lambda_{\tau}^0)$,  such that   there exist $\eps, \tilde\delta\in(0,1)$, and $\Omega'\in\sF_{\tau}$ with $\Omega'\subset \Omega_{\tau}$, $\bP(\Omega')>0$, satisfying a.e. on $\Omega'$,
{\small
\begin{align}
\essinf_{
s \in (\tau,(\tau+4\tilde\delta^2 )\wedge T],\,x\in B_{2\tilde\delta}(\xi_{\tau}) \cap \Lambda^{0,k,\alpha;\xi_{\tau}}_{\tau,s\wedge T}	
			} \!\!\!\!\!
	E_{\sF_{\tau}}
		\{ -\mathfrak{d}_{s}\phi(s,x)
		-
		\bH(s,x,\nabla \phi(s,x),\nabla^2\phi(s,x),\mathfrak{d}_{\omega}\nabla\phi(s,x))
		\}  \geq 2\,\eps.
		\label{eq-ex-1}
\end{align}
}

 Denote by $\hat{\tau}_k$ the stopping time associated to $\phi\in \underline\cG V(\tau,\xi_{\tau};\Omega_{\tau},k,\alpha)$. Note that we may think of $\xi$ valued in $\Lambda^0_T$ with $\xi(t)=\xi_{\tau}({t\wedge \tau})$ for all $t\in [0,T]$.  Moreover, associated to $\phi\in\mathscr C_{\sF}^2$, there is a partition: $0=\underline t_0<\underline t_1<\ldots<\underline t_n=T$. W.l.o.g., we assume $\Omega'=\{[\tau,\tau+4\tilde{\delta}^2] \subset [\underline t_j,\underline t_{j+1})\}=\Omega$ for some $j\in\{0,\ldots, n-1\}$. 

By assumption (ii) of $(\cA 1)$ and the measurable selection theorem, there exists $\bar\theta\in \cU$ such that a.s.,
$$
- \mathscr L^{\bar\theta(s)}\phi(s,\xi_s)  -f(s,\xi_s,\bar\theta(s))   \geq   
-\mathfrak{d}_{s}\phi(s,\xi_s)
-
\bH(s,\xi_s,\nabla \phi(s,\xi_s),\nabla^2\phi(s,\xi_s),\mathfrak{d}_{\omega}\nabla\phi(s,\xi_s))
 -       \eps,
$$
for almost all $s$ satisfying $\tau\leq s<(\tau+4\tilde \delta^2) \wedge T$. This together with \eqref{eq-ex-1} implies 
\begin{align}
\essinf_{\tau < s<\left(\tau+4\tilde\delta^2\right)\wedge T} 
	E_{\sF_{\tau}}
		\{ - \mathscr L^{\bar\theta(s)}\phi(s,\xi_s)  -f(s,\xi_s,\bar\theta(s)) \}  \geq \eps,\quad \text{a.s..}
		\label{thm-existence-eps-1}
\end{align}
 
Define
$\tau^{\bar \theta}_{k,\alpha}=\inf\{s> \tau; \|X^{\tau,\xi_{\tau};\bar\theta}\|_{\tau,s;\alpha} >k   \} \wedge T,
$ 
 and for each $\delta \in (0,\tilde\delta)$, set 
$$\tau^{\bar\theta}_{\delta}=\inf\{ s>\tau: X^{\tau,\xi_{\tau};{\bar\theta}}_s \notin B_{\delta}(\xi_{\tau}) \} \wedge T.$$
Then $\bP(\tau^{\bar\theta}_{\delta}   \wedge \tau^{\bar\theta}_{k,\alpha}>\tau)=1$. Putting $\bar \tau= \hat\tau_k \wedge \tau^{\bar\theta}_{\delta} \wedge \tau^{\bar\theta}_{k,\alpha} \wedge (\tau+\tilde \delta^2) \wedge T$, we have $\bP(\bar\tau >\tau)=1$.
Combining \eqref{thm-existence-eps-1} and the analysis in Remark \ref{rmk-bH} yields that 
\begin{align*}
\int_{\tau}^{  \bar\tau}\!\!\!
E_{\sF_{\tau}} \left[ 
 \mathscr L^{\bar \theta(s)} \phi (s,X_s^{\tau,\xi_{\tau};\bar\theta})  +f(s,X_s^{\tau,\xi_{\tau};\bar\theta},\bar\theta(s)) 
 	\right] 
	ds
\leq
	\int_{\tau}^{  \bar\tau} \!\!\!
	E_{\sF_{\tau}} \left[
	L^{\phi}_{\alpha'} | \delta ^{\alpha'} + \delta|    - \eps \right] ds, 
\end{align*}
where $\alpha'\in (0,1)$ is the exponent associated to $\phi\in\mathscr C^2_{\sF}$.


Choose a small $\delta \in (0,\tilde \delta)$ such that  $ 2 L^{\phi}_{\alpha'} | \delta ^{\alpha'} + \delta|  <  \eps$. Using the dynamic programming principle of Theorem \ref{thm-DPP} and the  It\^o-Kunita-Wentzell-Krylov formula of Lemma \ref{lem-ito-wentzell}, we have  
\begin{align}
0&\leq 
	E_{\sF_{\tau}}\left[ 
		  (\phi-V)\left( \bar\tau,X_{ \bar\tau}^{\tau,\xi_{\tau};\bar\theta}\right)  - (\phi-V)(\tau,\xi_{\tau}) 
		  \right] 
\nonumber\\
&\leq
E_{\sF_{\tau}} \left[ 
		    \phi\left( 		\bar\tau ,X_{ \bar\tau}^{\tau,\xi_{\tau};\bar\theta}\right) -\phi(\tau,\xi_{\tau})
		   +\int_{\tau}^{ \bar\tau} 
		   	f(s,X_s^{\tau,\xi_{\tau};\bar\theta},\bar\theta(s))\,ds    \right]
\nonumber\\
&\leq
E_{\sF_{\tau}} \left[ 
		     \int_{\tau}^{ \bar\tau} 
		   	\left( \mathscr L^{\bar \theta(s)} \phi (s,X_s^{\tau,\xi_{\tau};\bar\theta}) 
				+	f(s,X_s^{\tau,\xi_{\tau};\bar\theta},\bar\theta(s)) \right) 
							\,ds    \right]
\nonumber\\
&
\leq
	E_{\sF_{\tau}} \left[ 
	\int_{\tau}^{  \bar\tau} \!\!
	\left( L^{\phi}_{\alpha'} | \delta ^{\alpha'} + \delta|    - \eps\right) \, ds\right]
\nonumber
\\
&<
-\frac{  \eps}{2} E_{\sF_{\tau}} [\bar\tau-\tau],
\end{align}
which gives rise to a contradiction. Hence, $V$ is a viscosity subsolution of SPHJB equation \eqref{SHJB}.\medskip

\textbf{Step 2.}
We prove that $V$ is a viscosity supersolution of \eqref{SHJB}.  To the contrary, assume that for each  $(k,\alpha)\in \bN^+\times (0,\frac{1}{8})$ with $k\geq K_0$ and $\alpha\leq \alpha_0$ for some existing $(K_0,\alpha_0)\in\bN^+\times (0,1)$, there exists $\phi\in \overline\cG V(\tau,\xi_{\tau};\Omega_{\tau},k,\alpha)$ with $\tau\in \mathcal T^0$, $\Omega_{\tau}\in\sF_{\tau}$, $\mathbb P(\Omega_{\tau})>0$, and $\xi_{\tau}\in L^0(\Omega_{\tau},\sF_{\tau};\Lambda_{\tau}^0)$
such that 
there exist $\eps, \tilde\delta \in (0,1)$ and $\Omega'\in\sF_{\tau}$ with $\Omega'\subset \Omega_{\tau}$, $\bP(\Omega')>0$, satisfying a.e. on $\Omega'$,
\begin{align*}
\esssup_{
s \in (\tau,(\tau+4\tilde\delta^2) \wedge T],\,x\in B_{2\tilde\delta}(\xi_{\tau}) \cap \Lambda^{0,k;\xi_{\tau}}_{\tau,s\wedge T}	
			} \!\!\!\!\!
	E_{\sF_{\tau}}
		\{ -\mathfrak{d}_{s}\phi(s,x)
		-
			\bH(s,x,\nabla \phi(s,x),\nabla^2\phi(s,x),\mathfrak{d}_{\omega}\nabla\phi(s,x))
		\}  \leq -\eps.
\end{align*}
 
 Denote by $\hat{\tau}_k$ the stopping time associated to $\phi\in \overline\cG V(\tau,\xi_{\tau};\Omega_{\tau},k,\alpha)$. Again, we think of $\xi$ valued in $\Lambda^0_T$ with $\xi(t)=\xi_{\tau}(t\wedge \tau)$ for all $t\in [0,T]$, and associated to $\phi\in\mathscr C_{\sF}^2$, there is a partition: $0=\underline t_0<\underline t_1<\ldots<\underline t_n=T$. W.l.o.g., we assume $\Omega'=\{[\tau,\tau+4\tilde{\delta}^2] \subset [\underline t_j,\underline t_{j+1})\}=\Omega$ for some $j\in\{0,\ldots, n-1\}$. 

For each $\theta\in\cU$, define
$\tau^{\theta}=\inf\left\{ s>\tau: X^{\tau,\xi_{\tau};{\theta}}_s \notin B_{\tilde\delta}(\xi_{\tau}) \right\}$, and set 
$$
\tau^{ \theta}_{k,\alpha}=\inf\left\{s> \tau; \|X^{\tau,\xi_{\tau};\theta}\|_{\tau,s;\alpha} >k   \right\}, \quad \text{and } \bar\tau^{\theta} = \tau^{\theta} \wedge \tau^{ \theta}_{k,\alpha}.
$$  
Then $\bP( \bar \tau^{\theta}  >\tau)=1$. Recalling $\alpha\in (0,\frac{1}{8})$, we have for each $h\in (0, \frac{\tilde \delta ^2}{4})$,
 \begin{align}
 E_{\sF_{\tau}}\left[1_{\{\bar \tau^{\theta}     <\tau+h\}}\right]
 & 
 \leq E_{\sF_{\tau}}\left[1_{\{\tau^{\theta}     <\tau+h\}}\right] 
 	+
	E_{\sF_{\tau}}\left[1_{\{  \tau^{\theta}_{k,\alpha}   <\tau+h\}}\right] 
 \nonumber\\
 &=
 	E_{\sF_{\tau}}\left[1_{\{ \max_{\tau\leq s\leq \tau +h} |X^{\tau,\xi_{\tau};\theta}(s) -\xi(\tau)| + \sqrt{h}> {\tilde\delta} \}}\right]  
	+
	E_{\sF_{\tau}}\left[
		1_{
			\{ 
			\|X^{\tau,\xi_{\tau};\theta}\|_{\tau,\tau+h;\alpha} > k
			\}
			}
			\right] 
\nonumber\\
&
 \leq 
 	\frac{1}{(\tilde\delta-\sqrt{h})^8}  E_{\sF_{\tau}} \left[ \max_{\tau\leq s\leq \tau +h} |X^{\tau,\xi_{\tau};\theta} (s) -\xi(\tau)|^8\right]
	+
	\frac{1}{k^{16}} E_{\sF_{\tau}}
					\left[
					\|X^{\tau,\xi_{\tau};\theta}\|_{\tau,\tau+h;\alpha}^{16}
					\right]
\nonumber\\
&
\leq 
	\frac{K^8}{(\tilde\delta-\sqrt{h})^8}   (h+\sqrt{h})^8  
	+\frac{h^4}{k^{16}} E_{\sF_{\tau}}   
					\left[
					\|X^{\tau,\xi_{\tau};\theta}\|_{\tau,\tau+\tilde \delta;\alpha+\frac{1}{4}}^{16}
					\right]
	\nonumber \\
&
\leq
	\frac{256 \cdot K^8}{\tilde\delta ^8}   (h+\sqrt{h})^8   + \frac{C h^4}{k^{16}}
	\nonumber\\
&\leq
\tilde{C} h^4
	\quad \text{ a.s..}
	\label{est-tau}
 \end{align}
Here, the constants $K$ and $C$, independent of $(\theta,h,k,\tau)$, are from Lemma \ref{lem-SDE} and  Remark \ref{rmk-holder-contin}, respectively, and thus, the positive constant $\tilde C$ is independent from $\theta,h$, and $\tau$.

In view of Remark \ref{rmk-bH}, Lemma \ref{lem-ito-wentzell}, Theorem \ref{thm-DPP}, and the estimate \eqref{est-tau}, we have for each $h\in(0,\tilde\delta^2/4)$,
{\small
\begin{align*}
0
&=
	\frac{V(\tau,\xi_{\tau})-\phi(\tau,\xi_{\tau})}{h}
\\
&=
	\frac{1}{h} \essinf_{\theta\in\cU} E_{\sF_{\tau}}\left[
	\int_{\tau}^{\hat{\tau}_k \wedge (\tau+h)} 
		f(s,X_s^{\tau,\xi_{\tau};\theta},\theta(s))\,ds
		+V\left(\hat{\tau}_k \wedge (\tau+h), X_{\hat{\tau}_k\wedge (\tau+h)}^{\tau,\xi_{\tau};\theta}\right) -\phi(\tau,\xi_{\tau})
		\right]
\\
&\geq
	\frac{1}{h} \essinf_{\theta\in\cU} E_{\sF_{\tau}}\left[
	\int_{\tau}^{\hat{\tau}_k \wedge (\tau+h)} 
		f(s,X_s^{\tau,\xi_{\tau};\theta},\theta(s))\,ds
		+\phi\left(\hat{\tau}_k \wedge (\tau+h), X_{\hat{\tau}_k\wedge (\tau+h)}^{\tau,\xi_{\tau};\theta}\right) -\phi(\tau,\xi_{\tau})
		\right]
\\
&=
	\frac{1}{h} \essinf_{\theta\in\cU} E_{\sF_{\tau}}\left[
	\int_{\tau}^{\hat{\tau}_k \wedge (\tau+h)} 
		\left( \mathscr L^{\theta(s)}\phi\left(s,X_s^{\tau,\xi_{\tau};\theta}\right)+f(s,X_s^{\tau,\xi_{\tau};\theta},\theta(s))\right)\,ds
		\right]
\\
&\geq
	\frac{1}{h} \essinf_{\theta\in\cU} E_{\sF_{\tau}}\bigg[
		\int_{\tau}^{\bar\tau^{\theta}\wedge (\tau+h)\wedge \hat\tau_k} \bigg(
		\mathscr L^{\theta(s)}\phi\left(s,X_s^{\tau,\xi_{\tau};\theta} \right)+f(s,X_s^{\tau,\xi_{\tau};\theta},\theta(s))
      		  \bigg) \,ds\\
&	\quad\quad\quad
		  -1_{\{\hat\tau_k >\bar\tau^{\theta}\}\cap\{\tau+h>		\bar\tau^{\theta}\}}   \int_{\tau}^{\hat{\tau}_k \wedge (\tau+h)} \left| 
	\mathscr L^{\theta(s)}\phi\left(s,X_s^{\tau,\xi_{\tau};\theta} \right)+f(s,X_s^{\tau,\xi_{\tau};\theta},\theta(s))
      		  \right| \,ds
		\bigg]
\\
&\geq
	\frac{1}{h} \essinf_{\theta\in\cU} E_{\sF_{\tau}}\bigg[
		\int_{\tau}^{(\tau+h) \wedge \hat\tau_k} \bigg(
		\mathscr L^{\theta(s)}\phi\left(s,X_{s\wedge\bar \tau^{\theta}}^{\tau,\xi_{\tau};\theta} \right)+f(s,X_{s\wedge\bar \tau^{\theta}}^{\tau,\xi_{\tau};\theta},\theta(s))
      		  \bigg) \,ds\\
& \quad
-  1_{\{ \hat\tau_k > \bar\tau^{\theta}\}\cap \{\tau+h>\bar \tau^{\theta}\}   }
 \int_{\tau}^{\hat{\tau}_k \wedge (\tau+h)} \left| 
	\mathscr L^{\theta(s)}\phi\left(s,X_{s\wedge \bar \tau^{\theta}}^{\tau,\xi_{\tau};\theta} \right)+f(s,X_{s\wedge \bar \tau^{\theta}}^{\tau,\xi_{\tau};\theta},\theta(s))
      		  \right| \,ds
\\
&	\quad
		  - 1_{\{\hat\tau_k >\bar \tau^{\theta}\}\cap\{\tau+h>\bar \tau^{\theta}\}}  	   
		  \int_{\tau}^{\hat{\tau}_k \wedge (\tau+h)} \left| 
	\mathscr L^{\theta(s)}\phi\left(s,X_s^{\tau,\xi_{\tau};\theta} \right)+f(s,X_s^{\tau,\xi_{\tau};\theta},\theta(s))
      		  \right| \,ds
		\bigg]
\\
& \geq
	\eps\cdot E_{\sF_{\tau}} \left[
	\frac{(\hat\tau_k \wedge(\tau+h))-\tau}{h}		\right] 
	-\frac{1}{h} \esssup_{\theta\in\cU} \left( E_{\sF_{\tau}}\left[ 1_{\{\tau+h>\bar\tau^{\theta}\}}
				\right]\right)^{1/2}
	\\
&	\quad\quad
		  \cdot \left( E_{\sF_{\tau}} \left| \int_{\tau}^{\hat{\tau}_k \wedge (\tau+h)} \left| 
	\mathscr L^{\theta(s)}\phi\left(s,X_{s\wedge \bar \tau^{\theta}}^{\tau,\xi_{\tau};\theta} \right)+f(s,X_{s\wedge \bar \tau^{\theta}}^{\tau,\xi_{\tau};\theta},\theta(s))
      		  \right| \,ds \right|^2
		\right)^{1/2}\\
&
	-\frac{1}{h} \esssup_{\theta\in\cU} \left( E_{\sF_{\tau}}\left[  
		1_{\{\tau+h>\bar\tau^{\theta} \}}		\right]\right)^{1/2}
		\left( E_{\sF_{\tau}} \left| \int_{\tau}^{\hat{\tau}_k \wedge (\tau+h)} \!\! \left| 
	\mathscr L^{\theta(s)}\phi\left(s,X_{s }^{\tau,\xi_{\tau};\theta} \right)+f(s,X_{s }^{\tau,\xi_{\tau};\theta},\theta(s))
      		  \right| \,ds \right|^2
		\right)^{1/2}
\\
& \geq
\eps\cdot E_{\sF_{\tau}} \left[
\frac{(\hat\tau_k \wedge(\tau+h))-\tau}{h}\right]
- 2 h |\tilde C|^{1/2}   E_{\sF_{\tau}} \left[  \left( (\hat{\tau}_k \wedge (\tau+h)-\tau)
\int_{\underline t_j}^{\underline t_{j+1}-\tilde \delta^2} \big|\zeta^{\phi}_s\big|^2\,ds
  \right]  \right)^{1/2}\\
  &\rightarrow \eps,\quad \text{as }h\rightarrow 0^+,
\end{align*}
}
which results in a contradiction. Hence, $V$ is a viscosity supersolution.  
\end{proof}

\section{Uniqueness}

Recall that  $\overline{\mathscr V}$ is the set of all the classical supersolutions of SPHJB equation \eqref{SHJB} and $\underline{\mathscr V}$ the set of all the classical subsolutions, and by Lemma \ref{lem-classical-soltn}, both $\overline{\mathscr V}$ and $\underline{\mathscr V}$ are nonempty.  Set 
\begin{align*}
\overline{u}=\essinf_{\phi\in \overline{\mathscr V}} \phi, \quad 
\underline{u}=\esssup_{\phi\in \underline{\mathscr V}} \phi.
\end{align*}
Letting $V$ be the value function defined in \eqref{eq-value-func}, we shall prove the uniqueness results: (i) a classical solution must be $V$ if it exists; (ii)  $V$ is the unique viscosity solution approximated by classical supersolutions from above and by classical subsolutions from below.
\subsection{Comparison relation $\overline{u} \geq V  \geq \underline{u}$ and uniqueness of classical solution under Assumption $(\cA 1)$}

\begin{thm}\label{thm-comparison}
Letting Assumption $(\cA 1)$ hold, we have $\overline{u} \geq V  \geq \underline{u}$, i.e., for all $t\in [0,T]$ and $\xi_t\in \Lambda_t$, there holds $\overline{u}(t,\xi_t) \geq V(t,\xi_t)  \geq \underline{u}(t,\xi_t)$ a.s..
\end{thm}
 \begin{proof}
 \textbf{Step 1.}  To prove $V  \geq \underline{u}$, we need only verify the relation $\phi \leq  V $ for each $\phi\in  \underline{\mathscr V}$. Recall that, associated to $\phi\in\mathscr C_{\sF}^2$, there is a partition: $0=\underline t_0<\underline t_1<\ldots<\underline t_n=T$. Also, we have $\phi(T,x) \leq G(x)$ for all $x\in\Lambda_T$ a.s. and  for  each $s\in[0,T)$ with $y_s\in\Lambda_s$, there holds
\begin{align*}
 \text{ess}
 \limsup_{(\tau,x)\rightarrow (s^+,y_s)}
	E_{\sF_{s}} \left\{ 
	 -\mathfrak{d}_{s} \phi(\tau,x)-\bH(\tau,x,\nabla \phi(\tau,x),\nabla^2 \phi(\tau,x), \mathfrak{d}_{\omega}\nabla \phi(\tau,x) ) 
	 \right\} 
	\leq 0\text{, a.s.,}\\
 \text{i.e., }\quad\quad\quad\quad
 \text{ess}
 \limsup_{(\tau,x)\rightarrow (s^+,y_s)}
	E_{\sF_{s}} \left[\sup_{v\in U}\left\{ 
	  -\mathscr L^{v} \phi(\tau,x)-f(\tau,x,v)
	 \right\} \right]
	\leq 0\text{, a.s..}
\end{align*}
Then, for each $\theta\in \cU$, and $\xi_t\in \Lambda_t$ with $\underline t_{n-1}\leq t \leq \tau<T$,  we have by Lemma \ref{lem-ito-wentzell},
\begin{align*}
\phi(t,\xi_t)
&
= E_{\sF_t} \bigg[ 
\phi(\tau, X_{\tau}^{t,\xi_t;\theta})+
		\int_t^{\tau}\left( 
				-\mathscr L^{\theta(s)} \phi(s,X_s^{t,\xi_t;\theta}) 
				\right) ds
		\bigg]
\\
&\leq
 E_{\sF_t} \bigg[ 
\phi(\tau, X_{\tau}^{t,\xi_t;\theta})+
		\int_t^{\tau}
				 f(s,X_s^{t,\xi_t;\theta},\theta(s))
				 ds
		\bigg]
\\
&\rightarrow E_{\sF_t} \bigg[ 
\phi(T, X_{T}^{t,\xi_t;\theta})+
		\int_t^{T}
				 f(s,X_s^{t,\xi_t;\theta},\theta(s))
				 ds
		\bigg] ,\quad\text{as }\tau \rightarrow T.
\end{align*}
Thus, $\phi(t,\xi_t) \leq J(t,\xi_t;\theta)$ as $\phi(T,x_T) \leq G(x_T)$ a.s. for all $x_T\in \Lambda_T$ and this together with the arbitrariness of $\theta$ implies that  $\phi(t,\xi_t) \leq V(t,\xi_t)$ a.s. for  $t\in [\underline t_{n-1},T)$.  Similarly, we may verify $\phi \leq V$ recursively over the time intervals $[\underline t_{n-2},\underline t_{n-1})$, \ldots, $[0,\underline t_{1})$.
 
%
 
 \textbf{Step 2.}  We prove $   \overline{u}\geq V$. For each $\phi\in  \overline{\mathscr V}$, recall that $\phi(T,x) \geq G(x)$ for all $x\in\Lambda_T$ a.s. and  for  each $s\in[0,T)$ with $y_s\in\Lambda_s$, there holds
\begin{align}
 \text{ess}
 \liminf_{(\tau,x)\rightarrow (s^+,y_s)}
	E_{\sF_{s}} \left\{ 
	 -\mathfrak{d}_{s} \phi(\tau,x)-\bH(\tau,x,\nabla \phi(\tau,x),\nabla^2 \phi(\tau,x), \mathfrak{d}_{\omega}\nabla \phi(\tau,x) ) 
	 \right\} 
	\geq 0\text{, a.s.,} 
	\nonumber\\
\text{i.e., }\quad\quad\quad\quad
 \text{ess}
 \liminf_{(\tau,x)\rightarrow (s^+,y_s)}
	E_{\sF_{s}} \left[\sup_{v\in U}\left\{ 
	  -\mathscr L^{v} \phi(\tau,x)-f(\tau,x,v)
	 \right\} \right]
	\geq 0\text{, a.s..} \label{thm-unique-1-L}
\end{align}
Also, associated to $\phi\in\mathscr C_{\sF}^2$, there is a partition: $0=\underline t_0<\underline t_1<\ldots<\underline t_n=T$.

As in Step 1, we first prove the comparison on the interval $[\underline t_{n-1},T)$.
For each $t\in [\underline t_{n-1},T)$ and $\xi_t\in L^0(\Omega,\sF_t;\Lambda_t)$, we may extend $\xi_t$ to be valued in $\Lambda_T$ with $\xi(s)=\xi_t(t\wedge s)$ for all $s\in [0,T]$. Take $\delta\in (0,(T-t)\wedge 1)$. Then by \eqref{thm-unique-1-L} and the measurable selection theorem, for each $\eps\in (0,1)$, there exists $\bar \theta \in\cU$ such that for all $N\in \bN^+\setminus \{1,2\}$, it holds that
{\small
\begin{align}
\quad E_{\sF_{t}} \left[
\int_t^{t+\frac{\delta}{N}}\!\!\!
  \left\{ 
	  -\mathscr L^{\bar\theta(s)} \phi(s,\xi_s)-f(s,\xi_s,\bar\theta(s))
	 \right\} \,ds 
	 \right]
&\geq \int_t^{t+\frac{\delta}{N}} \!\!\!
E_{\sF_{t}} \left[\sup_{v\in U}\left\{ 
	  -\mathscr L^{v} \phi(s,\xi_s)-f(s,\xi_s,v)
	 \right\} -\eps  \right]\,ds
\nonumber\\
 &
	  \geq \frac{-\delta \eps}{N}, \quad \text{a.s..} \label{thm-unique-1}
\end{align}
}
In view of Remark \ref{rmk-bH} and (iii) of Lemma \ref{lem-SDE}, we have 
\begin{align}
&E_{\sF_{t}} \left[
\int_t^{t+\frac{\delta}{N}} \!\!
 \big| ( 
	  \mathscr L^{\bar\theta(s)} \phi(s,\xi_s)+f(s,\xi_s,\bar\theta(s))
	 )
	 -
	 ( 
	  \mathscr L^{\bar\theta(s)} \phi(s,X^{t,\xi_t;\bar\theta}_s)
	  +f(s,X^{t,\xi_t;\bar\theta}_s,\bar\theta(s))
	 )\big|
	  \,ds 
	 \right]
	 \nonumber
\\
&
\leq E_{\sF_{t}} \left[
\int_t^{t+\frac{\delta}{N}} \left| L^{\phi}_{\alpha} \left(\| X_s^{t,\xi_t;\bar\theta} -\xi_s\|_0^{\alpha} + \|X_s^{t,\xi_t;\bar\theta} -\xi_s\|_0\right)
\right|\,ds
\right]
\nonumber
\\
&
\leq C_1 \left(\frac{\delta}{N}\right)^{1+\frac{\alpha}{2}}, \label{thm-unique-2}
\end{align}
where $\alpha\in (0,1)$ is the exponent associated to $\phi \in \mathscr C^2_{\sF}$, the constant $L_{\alpha}^{\phi}$, depending on $\delta$, is associated to the interval $[\underline t_{n-1}, t+\delta]$, and $C_1$ depends only on $L_{\alpha}^{\phi}$ and $K$. 

Starting with the obtained $\left(t+\frac{\delta}{N}, X^{t,\xi_t;\bar\theta}_{t+\frac{\delta}{N}}\right)$, we may conduct the same discussions as in \eqref{thm-unique-1} and \eqref{thm-unique-2} recursively over the time interval $[t+\frac{\delta}{N}, t+\frac{2\delta}{N}]$, \ldots, $[t+\frac{(N-1)\delta}{N},t+\delta]$ for $N-1$ steps. Step by step, the controls $\bar\theta$s may be patched together, and there exists some $\bar \theta\in\cU$ such that 
{\small
\begin{align*}
&\phi(t,\xi_t)
\\
&
= E_{\sF_t} \bigg[ 
\phi(t+\delta,X_{t+\delta}^{t,\xi_t;\bar\theta})+
		\int_t^{t+\delta}\left( 
				-\mathscr L^{\bar\theta(s)} \phi(s,X_s^{t,\xi_t;\bar\theta}) 
				\right) ds
		\bigg]
\\
&
\geq E_{\sF_t} \bigg[ 
\phi(t+\delta,X_{t+\delta}^{t,\xi_t;\bar\theta})+
		\int_t^{t+\delta}\!\!
				f(s,X_s^{t,\xi_t;\bar\theta},\bar\theta(s)) 
				  ds
		\bigg]
-
E_{\sF_{t}} \left[
\int_t^{  {t+\delta}}\!\!
  \left\{ 
	  \mathscr L^{\bar\theta(s)} \phi(s,\xi_s)+f(s,\xi_s,\bar\theta(s))
	 \right\}  ds 
	 \right]
\\
&\quad
	-E_{\sF_{t}} \left[
\int_t^{  {t+\delta}} \!\!
 \big| ( 
	  \mathscr L^{\bar\theta(s)} \phi(s,\xi_s)+f(s,\xi_s,\bar\theta(s))
	 )
	 -
	 ( 
	  \mathscr L^{\bar\theta(s)} \phi(s,X^{t,\xi_t;\bar\theta}_s)
	  +f(s,X^{t,\xi_t;\bar\theta}_s,\bar\theta(s))
	 )\big|
	  \,ds 
	 \right]
 \nonumber\\
 &
 \geq E_{\sF_t} \bigg[ 
\phi(t+\delta,X_{t+\delta}^{t,\xi_t;\bar\theta})+
		\int_t^{t+\delta}\!\!
				f(s,X_s^{t,\xi_t;\bar\theta},\bar\theta(s)) 
				  ds
		\bigg] - \delta \eps - N \cdot C_1 \left(\frac{\delta}{N}\right)^{1+\frac{\alpha}{2}}
 \nonumber\\
 &
 =
 E_{\sF_t} \bigg[ 
\phi(t+\delta,X_{t+\delta}^{t,\xi_t;\bar\theta})+
		\int_t^{t+\delta}\!\!
				f(s,X_s^{t,\xi_t;\bar\theta},\bar\theta(s)) 
				  ds
		\bigg]
		-\delta \eps 
		 - \frac{ C_1\delta^{1+\frac{\alpha}{2}}} {N^{\frac{\alpha}{2}}}, \quad \text{a.s..}
\end{align*}
}
Choosing a big $N$ so that ${\eps \cdot N^{\frac{\alpha}{2}}}>C_1$, we have
{\small
\begin{align*}
\phi(t,\xi_t) 
&\geq J(t,\xi_t;\bar\theta) + E_{\sF_t} \left[ \phi(t+\delta,X_{t+\delta}^{t,\xi_t;\bar\theta}) -\phi(T,X_{T}^{t,\xi_t;\bar\theta})  \right]
-L(T-t-\delta)
- \eps \left\{ \delta^{1+\frac{\alpha}{2}} +\delta\right\}
\\
&\geq
V(t,\xi_t)
-E_{\sF_t} \left[
\sup_{s\in [T-\delta-t,T],x_s\in \Lambda_s} |\phi(s,x_s)-\phi(T,x_T)|
 \right]
 -L(T-t-\delta)
-\eps \left\{  \delta^{1+\frac{\alpha}{2}} +\delta\right\} 
\\
&\rightarrow V(t,\xi_t)-\eps \left\{ (T-t)^{1+\frac{\alpha}{2}} +(T-t)\right\}, \text{a.s., as }\delta \text{ tends to } T-t.
\end{align*}
} 
The arbitrariness of $(t,\xi_t,\eps)$ further implies that $V\leq \phi$ on $[\underline t_{n-1}, T)$, and recursively, the comparison may be verified over the time intervals $[\underline t_{n-2},\underline t_{n-1})$, \ldots, $[0,\underline t_{1})$.
 Finally, we obtain $V\leq \overline u$ over the whole time interval $[0,T]$.
 \end{proof}
 
 A straightforward application of Theorem \ref{thm-comparison} gives the uniqueness of classical solution.
 \begin{cor}\label{unique-classical-sltn}
 Let Assumption $(\cA 1)$ hold. If $u$ is a classical solution of SPHJB equation \eqref{SHJB}, then $u(t,\xi_t)=V(t,\xi_t)$, a.s. for all $t\in [0,T]$ and $\xi_t\in \Lambda_t$.
  \end{cor}

\subsection{Uniqueness: $\overline{u} = V =\underline{u}$ for superparabolic cases with state-dependent $\sigma$}

First, we  write the Wiener process $W=(\tilde{W},\bar{W})$, where $\tilde{W}$ and $\bar{W}$ are two mutually independent and respectively, $m_0$- and $m_1$($=m-m_0$)-dimensional  Wiener processes. In what follows, we adopt the decomposition $\sigma=(\tilde{\sigma},\bar{\sigma})$ with $\tilde{\sigma}$ and $\bar{\sigma}$ valued in $\bR^{d\times m_0}$ and $\bR^{d\times m_1}$ respectively associated to $\tilde{W}$ and $ \bar{W}$. 
Denote by $\{\tilde{\sF}_t\}_{t\geq0}$ the natural filtration generated by $\tilde{W}$ and augmented by all the
$\bP$-null sets. 

\bigskip\medskip
   $({\mathcal A} 2)$ (i)  For each $t\in[0,T]$,\, $x_t\in\Lambda_t$, $y_T\in\Lambda_T$,  and ${v}\in U$,  $G(y_T)$ is $\tilde{\sF}_T$-measurable and for the  random variables $g=\beta^i(t,x_t,v),f(t,x_t,v)$, $i=1,\ldots, d$,
$$
  g:~\Omega\rightarrow \bR
\text{ is } \tilde{\sF}_t \text{-measurable.}
$$
 (ii) For $i=1,\ldots,d$, $j=1,\ldots, m$, $\sigma^{ij}:~[0,T]\times\bR^d \times U \rightarrow\bR$ belongs to $C(U; C^2([0,T]; C^3(\bR^d)))$
 and there exists $\kappa\in(0,\infty)$  such that
   \begin{align*}
     \text{(Superparabolicity)}\quad \sum_{i,j=1}^d\sum_{k=1}^{m_1} \bar\sigma^{ik}\bar\sigma^{jk}(t,x,v)\xi^i\xi^j\geq \lambda |\xi|^2\quad \,\,\forall\, (t,x,v,\xi)\in [0,T]\times\bR^d \times U \times\bR^d.
   \end{align*}


 Moreover, we assume the following continuity properties of the coefficients $\beta,\,f$, and $G$.
\begin{enumerate}
\item [$(\cA 3)$] For each $\eps>0$, there exist partition $0=t_0<t_1<\cdots<t_{N-1}<t_N=T$ for some $N>3$ and functions $G^N\in C^{3}(\bR^{ N\times (d+m_0)+d};\bR)$,
$$
(f^N,\beta^N)\in C( U;C^3([0,T]\times\bR^{(m_0+d)\times N +d};\bR)) \times C(U;C^3([0,T]\times\bR^{(m_0+d)\times N +d};\bR^d)),
$$  
such that $G^{\eps}:=\esssup_{x\in\Lambda_T^0} \left|G^N( \tilde W({t_1}),\ldots, \tilde W({t_N}),x(t_0),\ldots,x(t_N))-G(x) \right|$, and 
\begin{align*}
& 
	f^{\eps}(t) :=\esssup_{(x,v)\in\Lambda_t^0\times U}
		\left|f^N( \tilde W({t_1\wedge t}),\ldots, \tilde W({t_N\wedge t}),t,x({t_0\wedge t}),\ldots, x({t_N\wedge t}),v)-f(t,x,v)\right|,
\\
&
 	\beta^{\eps}(t) :=\esssup_{(x,v)\in \Lambda_t^0 \times U}
		\left|\beta^N(\tilde  W({t_1\wedge t}),\ldots, \tilde W({t_N\wedge t}),t,x({t_0\wedge t}),\ldots, x({t_N\wedge t}),v)-\beta(t,x,v)\right|,
\end{align*}
 are $\tilde \sF_t$-adapted with
\begin{align*}
	\left\| G^{\eps}  \right\|_{L^2(\Omega,\sF_T;\bR)} + \left\| f^{\eps}  \right\|_{L^2(\Omega\times[0,T];\bR)}   + \left\| \beta^{\eps}  \right\|_{L^2(\Omega\times[0,T];\bR)}    <\eps,
\end{align*}
and $G^N$, $f^N$ and $\beta^N$ are uniformly Lipschitz-continuous in the space variable $x$ with an identical Lipschitz-constant $L_c$ independent of $N$ and $\eps$. 

\end{enumerate}
 
 \begin{rmk}
In Assumption $(\cA 3)$, the coefficients $\beta, f,$ and $G$ are approximated via regular functions.
Indeed, such approximations may be proved in a similar way to \cite[Lemma 4.2]{qiu2020controlled} if we assume the uniform time-continuity: for $g=\beta^i,f$, $i=1,\ldots,d$, there is a continuously increasing function $\rho: [0,\infty)\rightarrow [0,\infty)$ with $\rho(0)=0$ such that for all $r,t,s\in [0,T]$ with $s\leq t\leq r$, and $x\in \Lambda_T$,
{\small
\begin{align}
|g(r,(x_{t\wedge \cdot})_r)-g(r,(x_{s\wedge \cdot})_r) |+
|g(t,x_t)-g(s,x_s)| + |G((x_{t\wedge \cdot})_T)-G((x_{s\wedge \cdot})_T)| \leq \rho (|t-s|), \text{ a.s..} \label{A3-sufficient}
\end{align}
}
For example, the relation \eqref{A3-sufficient} is obviously satisfied if we take $g(t,x_t)=\int_0^t \phi(x_t(s))\,ds$ for $t\in[0,T]$, and $G(x_T)=\zeta\cdot \int_0^T \psi(x_T(s))\,ds$, where $\zeta\in L^{\infty}(\Omega,\tilde\sF_T;\bR)$ and $\phi$ and $\psi$ are bounded and uniformly continuous functions on $\bR^d$.

 \end{rmk}

\begin{thm}\label{thm-main}   
 Letting $(\cA 1)-(\cA 3)$ hold and $V$ be the value function in \eqref{eq-value-func}, we have $\overline{u} = V =\underline{u}$, i.e., for each $t\in[0,T]$ and $x_t\in \Lambda_t^0$, there holds $\overline{u} (t,x_t)= V(t,x_t) =\underline{u}(t,x_t)$ a.s..
 \end{thm}

\begin{proof}
By Theorem \ref{thm-comparison}, we have $\underline u\leq V\leq \overline u$.  Therefore, it is sufficient to construct functions from $\overline{\mathscr V}$ and $\underline{{\mathscr V}}$ to approximate the value function $V$ from above and from below respectively.  

For each $\eps\in(0,1)$, we take $(G^{\eps},\,f^{\eps},\,\beta^{\eps})$ and $(G^N,f^N,\beta^N)$ as in Assumption $(\cA 3)$. By the theory of backward SDEs (see \cite{Hu_2002} for instance), let the pair $(Y^{\eps},Z^{\eps})$   be the unique adapted solution  to backward SDE
$$
Y^{\eps}(s)=G^{\eps}+
	\int_s^T\left(f^{\eps}(t)+\tilde L \beta^{\eps}(t)\right)\,dt
		-\int_s^TZ^{\eps}(t)\,d \tilde W(t),
$$
where the constant $\tilde L$ is to be determined later. For each $s\in[0,T)$ and $x_s\in \Lambda_s$,  let
{\small
\begin{align*}
{V}^{\eps}(s,x_s)
=\essinf_{\theta\in\cU} E_{\sF_s} \bigg[
	\int_s^Tf^N\Big( & \tilde W({t_1\wedge t}),\ldots, \tilde W({t_N\wedge t}),t,X^{s,x_s;\theta,N}(0),X^{s,x_s;\theta,N}({t_1\wedge t}) ,\ldots, 
	\\
	&X^{s,x_s;\theta,N}({t_N\wedge t}), \theta(t)\Big)\,dt\\
		+G^N\Big( \tilde W({t_1}),\ldots,& \tilde W({t_N}),,X^{s,x_s;\theta,N}(0),X^{s,x_s;\theta,N}({t_1}),\ldots,X^{s,x_s;\theta,N}(t_N)\Big)
				\bigg],
\end{align*}
}
where $X^{s,x_s;\theta,N}$ satisfies the SDE
\begin{equation*}
\left\{
\begin{split}
dX(t)&=\beta^N(\tilde W({t_1\wedge t}),\ldots,\tilde W({t_N\wedge t}),t,X (0),X ({t_1\wedge t}) ,\ldots,X ({t_N\wedge t}),\theta(t))\, dt 
\\
&\quad 
+\sigma(t,X(t),\theta(t))\,dW(t)
	,\,\,\,t\in[s,T]; \\
 X(t)&=x_s(t),\quad t\in[0,s].
\end{split}
\right.
\end{equation*}

For each $s\in[t_{N-1},T)$, we have the representation
$$V^{\eps}(s,x_s)=\tilde V^{\eps}( \tilde W(t_1),\ldots, \tilde W({t_{N-1}}),{\tilde W}(s),s,x(0),\ldots,x(t_{N-1}),x(s))$$ 
with
{\small
\begin{align*}
&\tilde V^{\eps}( \tilde W(t_1),\ldots, \tilde W({t_{N-1}}),\tilde y,s,x(0),\ldots,x(t_{N-1}),\tilde x) \\
&=\essinf_{\theta\in\cU} E_{\sF_s,  \tilde W(s)=\tilde y,x_s(s)=\tilde x} \bigg[
	\int_s^Tf^N\left(  \tilde W({t_1}),\ldots, \tilde W({t_{N-1}}), \tilde W({ t}),t,\ldots,x(t_{N-1}),X^{s,x_s;\theta,N}(t),\theta(t)\right)\,dt\\
	&
		+G^N\left( \tilde W({t_1}),\ldots, \tilde W({t_N}),x(0),\ldots,x(t_{N-1}),X^{s,x_s;\theta,N}(T)\right)
		\bigg].
\end{align*}
}
 By the viscosity solution theory of fully nonlinear parabolic PDEs (see \cite[Theorems I.1 and II.1]{lions-1983} for instance),  the function $\tilde V^{\eps}( \tilde W(t_1),\cdots, \tilde W({t_{N-1}}),\tilde y,s,x(0),\cdots,x(t_{N-1}),\tilde x)$ 
satisfies the  following HJB equation over time interval $[t_{N-1},t_N]$:
{\small
\begin{equation*}
  \left\{\begin{array}{l}
  \begin{split}
  -D_tu(\tilde y, t,\tilde x)=\,& 
  \frac{1}{2} \text{tr}\left(D_{\tilde y \tilde y}u(\tilde y, t,\tilde x)\right) 
  + \essinf_{v\in U} \bigg\{   
      \text{tr}\Big(\frac{1}{2}\sigma\sigma'(t,\tilde x,v)D_{\tilde x\tilde x}u(\tilde y, t,\tilde x) + \tilde\sigma(t,\tilde x,v) D_{\tilde y\tilde x}u(\tilde y, t,\tilde x) \Big)\\
	&\quad\quad
		  +(\beta^N)'(  \tilde W({t_1}),\ldots,\tilde  W({t_{N-1}}) ,\tilde y,t,  x(0),\ldots,x(t_{N-1}),\tilde x,v)D_{\tilde{x}}u(\tilde y, t,\tilde x)\\
	  &\quad +f^N(  \tilde W({t_1}),\ldots, \tilde W({t_{N-1}}) ,\tilde y,t,  x(0),\ldots,x(t_{N-1}),\tilde x,v)\bigg\}
                     ;\\
    u(\tilde y, T,\tilde x)=\, &G^N(  \tilde W({t_1}),\ldots, \tilde W({t_{N-1}}) ,\tilde y,  x(0),\ldots,x(t_{N-1}),\tilde x).
        \end{split}
  \end{array}\right.
\end{equation*}
}
Here, we just write $x(t_j)=x_s(t_j)$ for $j=0,\ldots,N-1$, as they are deemed to be fixed for $s\in (t_{N-1},T]$; the classical derivatives are denoted by $D_{\tilde{x}}, D_{\tilde x\tilde x},D_{\tilde y\tilde x}, D_t$, and $D_{\tilde y \tilde y}$. 
The regularity theory of viscosity solutions then
implies that for each $(x(0),\cdots,x(t_{N-1}))\in\bR^{N\times d}$,
\begin{align}
&\tilde V^{\eps}( \tilde W(t_1),\ldots, \tilde W({t_{N-1}}),\cdot,\cdot,x(0),\ldots,x(t_{N-1}),\cdot)
\nonumber\\
&\quad\in 
\cap_{\bar t\in (t_{N-1},T)} L^{\infty}\left(\Omega,\tilde \sF_{t_{N-1}}; C^{1+\frac{\bar\alpha}{2},2+\bar\alpha}([t_{N-1},\bar t\, ]\times\bR^{m_0+d}) \right) 
, \label{est-vis-solution}
\end{align}
for some $\bar\alpha \in (0,1)$, where the \textit{time-space} H\"older space $C^{1+\frac{\bar\alpha}{2},2+\bar\alpha}([t_{N-1},\bar t\,]\times\bR^d)$ is defined as usual. At time $t_{N-1}$, we check that $ V^{\eps}$ is still uniformly Lipschitz-continuous w.r.t. $( \tilde W(t_1)$, $\ldots$, $\tilde W({t_{N-1}}) )$ and $(x(0),\ldots,x(t_{N-1}))$. Then, we may conduct the same discussions on time interval $[t_{N-2},t_{N-1}]$ with the previously obtained $V^{\eps}(t_{N-1},x)$ as the terminal value, and recursively on intervals $[t_{N-3},t_{N-2}]$, $\dots$, $[0,t_{1}]$.

Meanwhile, applying the  It\^o-Kunita formula of \cite[Pages 118-119]{kunita1981some} to $\tilde{V}^{\eps} $ on $[t_{N-1},T]$ yields
 {\small
 \begin{equation}\label{SHJB-N}
  \left\{\begin{array}{l}
  \begin{split}
  &-dV^{\eps}(t,x_t)\\ 
&= \essinf_{v\in U} \bigg\{ \text{tr}\left( \frac{1}{2}\sigma\sigma' D_{\tilde x\tilde x}\tilde{V}^{\eps} + \tilde \sigma D_{\tilde y\tilde x}\tilde{V}^{\eps} \right) ( \tilde W({t_1}),\ldots, \tilde W(t),t,x(0),\ldots,x(t_{N-1}),x(t),v)   \\
&\quad\quad
+
  (\beta^N)'(\tilde  W({t_1}),\ldots,\tilde W({t_{N-1}}), \tilde W(t),t,x(0),\ldots,x(t_{N-1}) ,x(t),v) \nabla V^{\eps}(t,x_t)
  \\
&\quad\quad	
  +f^N( \tilde W({t_1}),\ldots,\tilde W({t_{N-1}}), \tilde W(t),t,x(0),\ldots,x(t_{N-1}),x(t),v)\bigg\}\,dt\\
&\quad\quad
		-D_{\tilde y} \tilde{V}^{\eps}( \tilde W({t_1}),\ldots, \tilde W(t),t,x(0),\ldots,x(t_{N-1}),x(t)) \,d \tilde W(t), 
                     \quad 
                     t\in [t_{N-1},T) \text{ and }x\in\Lambda_t;\\
    &V^{\eps}(T,x_T)=\, G^N(  \tilde W({t_1}),\ldots, \tilde W(T),x(0),\ldots,x(t_{N-1}) ,x(T) ), \quad x_T\in\Lambda_T.
    \end{split}
  \end{array}\right.
\end{equation}
}
It follows similarly on intervals $[t_{N-2},t_{N-1})$, $\dots$, $[0,t_{1})$. 
Subsequently, we show $V^{\eps}  \in \mathscr C^2_{\sF}$ for which in view of the regularity of function $\tilde V^{\eps}$ in  \eqref{est-vis-solution},  we need only verify the existence of $\mathfrak{d}_{\omega} \nabla V^{\eps}$ as required in (b) of (i) in Definition \ref{defn-testfunc}. Indeed, we shall prove 
\begin{align}
\left(\mathfrak{d}_{\omega}\nabla V^{\eps}\right)^{lj}=\left( D_{\tilde y\tilde x}\tilde{V}^{\eps}\right)^{ l j}, \quad \text{and}\quad \left(\mathfrak{d}_{\omega}\nabla V^{\eps}\right)^{ \tilde lj}=0, \text{ on time interval }(t_i,t_{i+1}),
\label{dwD-value}
\end{align}
for $l=1,\ldots, m_0$, $\tilde l=m_0+1,\ldots, m$,  $j=1,\ldots,d$, $i=0,\ldots,N-1$.  

Consider the subinterval $(t_{N-1},T)$. For each $[\tilde t_{N-1}, \tilde t_N]\subset (t_{N-1},T)$, denote by $\Pi=\{\tilde t_{N-1}= \tau_0<\ldots<\tau_{\tilde N}=\tilde t_{N}\}$ a subdivision of $[\tilde t_{N-1},\tilde t_{N}]$ with $|\Pi|=\max_{1\leq k\leq \tilde N}|\tau_k-\tau_{k-1}|$. For each $x\in \Lambda_T$, and $M^l(t)=\int_{t\wedge \tilde t_j}^t g(s)\,dW^l(s)$ for some $g\in L^{\infty}(\Omega\times [0,T];\sP) $, $l=1,\ldots,m$, we verify
{\small
\begin{align*}
&\int_{\tilde t_{N-1}}^{\tilde t_{N}} \left(\mathfrak{d}_{\omega} \nabla V^{\eps} \right)^{li} (t,x_t) g(t)\,dt
\\
&=
\lim_{|\Pi| \rightarrow 0^+} \sum_{k=0}^{\tilde N-1}  \left( \nabla_i V^{\eps}(\tau_{k+1},x_{\tau_k,\tau_{k+1}-\tau_k})- \nabla_i V^{\eps}(\tau_k,x_{\tau_k})\right) \int_{\tau_k}^{\tau_{k+1}} g(s)\,dW^l(s), \text{ in probability,}
\end{align*}
}
for $i=1,\ldots,d$, with $\mathfrak{d}_{\omega}\nabla V^{\eps}$ given by \eqref{dwD-value}. For simplicity, we write $\tilde V^{\eps}(\tilde y,t,\tilde x)$ over the time interval $[t_{N-1},T]$ for $\tilde V^{\eps}( \tilde W(t_1),\ldots, \tilde W({t_{N-1}}),\tilde y,t,x(0),\ldots,x(t_{N-1}),\tilde x)$. The computations are based on the estimate in \eqref{est-vis-solution} and the relation $V^{\eps}(t,x_t)=\tilde V^{\eps}(\tilde W(t), t,x(t))$ for $t\in (t_{N-1},T)$. Denote $\Delta \tilde W(\tau_k)=\tilde W(\tau_{k+1}) - \tilde W(\tau_k)$ for $k=0,\ldots,\tilde N-1$.  First comes the decomposition:
{\small
\begin{align*}
 	\nabla_i V^{\eps}(\tau_{k+1},x_{\tau_k,\tau_{k+1}-\tau_k})- \nabla_i V^{\eps}(\tau_k,x_{\tau_k})
 &
 =
 D_{\tilde x^i} \tilde V^{\eps}(\tilde W (\tau_{k+1}), \tau_{k+1},x(\tau_k) )- D_{\tilde x^i} \tilde V^{\eps}(\tilde W (\tau_{k}), \tau_{k},x(\tau_k) )
 \\
 &
 = D_{\tilde x^i} \tilde V^{\eps}(\tilde W (\tau_{k+1}), \tau_{k+1},x(\tau_k) )- D_{\tilde x^i} \tilde V^{\eps}(\tilde W (\tau_{k}), \tau_{k+1},x(\tau_k) )
 \\
 &\quad\quad
 	+
	D_{\tilde x^i} \tilde V^{\eps}(\tilde W (\tau_{k}), \tau_{k+1},x(\tau_k) )- D_{\tilde x^i} \tilde V^{\eps}(\tilde W (\tau_{k}), 
	\tau_{k},x(\tau_k) )
\\
&:=F_1+F_2.
\end{align*}
}
Applying the integration-by-parts formula gives
{\small
\begin{align*}
F_1
&=
	 \sum_{j=1}^{m_0}\int_0^1 D_{\tilde y^j\tilde x^i}  \tilde V^{\eps} (\tilde W (\tau_{k}) +\lambda \Delta \tilde W (\tau_{k}) , \tau_{k+1},x(\tau_k))\,d\lambda \cdot (\Delta \tilde W (\tau_{k}))^{j}
\\
&= \sum_{j=1}^{m_0}D_{\tilde y^j\tilde x^i} \tilde V^{\eps} (\tilde W (\tau_{k})   , \tau_{k+1},x(\tau_k))  \cdot (\Delta \tilde W (\tau_{k}))^{j}
+ \eps_k^1, \text{ a.s.,}
\end{align*}
}
with 
{\small
\begin{align}
&|\eps_k^1|
\nonumber\\
&=\left|
 \sum_{j=1}^{m_0} \int_0^1 \left( D_{\tilde y^j\tilde x^i}  \tilde V^{\eps} (\tilde W (\tau_{k}) +\lambda \Delta \tilde W (\tau_{k}) , \tau_{k+1},x(\tau_k))
 - D_{\tilde y^j\tilde x^i}  \tilde V^{\eps} (\tilde W (\tau_{k})   , \tau_{k+1},x(\tau_k)) 
   \right)d\lambda  \cdot (\Delta \tilde W (\tau_{k}))^{j}
	\right|
\notag
\\
&\leq C
 \sum_{j=1}^{m_0} \int_0^1 |\lambda \Delta \tilde W (\tau_{k}) |^{\bar \alpha}  d\lambda  \cdot \left| (\Delta \tilde W (\tau_{k}))^{j}\right|
\notag
\\
&\leq C_1 |\Delta \tilde W (\tau_{k})|^{1+\bar\alpha}, \text{a.s.,} 
\label{eps-1}
\end{align}
}
where we have used the H\"older estimate \eqref{est-vis-solution}. For $F_2$, we notice that
{\small
\begin{align}
&  \frac{ \tilde V^{\eps}\Big(\tilde W (\tau_{k}), 
	\tau_{k},x(\tau_k) + e_i {|\Delta \tilde W (\tau_{k})|} \Big)
	- \tilde V^{\eps}\Big(\tilde W (\tau_{k}), 
	\tau_{k},x(\tau_k)  \Big)}{{|\Delta \tilde W (\tau_{k})|}}
	-D_{\tilde x^i} \tilde V^{\eps}(\tilde W (\tau_{k}), 
	\tau_{k},x(\tau_k) )
\nonumber\\
&= 
\int_0^1\left(
D_{\tilde x^i} \tilde V^{\eps}(\tilde W (\tau_{k}), 
	\tau_{k},x(\tau_k) + \lambda e_i {|\Delta \tilde W (\tau_{k})|} )
- D_{\tilde x^i}\tilde V^{\eps}(\tilde W (\tau_{k}), 
	\tau_{k},x(\tau_k) )
\right) d\lambda
\nonumber\\
&=
\int_0^1\int_0^1
 D_{\tilde x^i\tilde x^i} \tilde V^{\eps}(\tilde W (\tau_{k}), 
	\tau_{k},x(\tau_k) + \lambda_1 \lambda e_i {|\Delta \tilde W (\tau_{k})|} ) \,\lambda d\lambda_1
  d\lambda \cdot |\Delta \tilde W (\tau_{k})|,\text{ a.s.,}
  \label{difference-1}
\end{align}
}
and similarly,
{\small
\begin{align}
&  \frac{ \tilde V^{\eps}\Big(\tilde W (\tau_{k}), 
	\tau_{k+1},x(\tau_k) + e_i {|\Delta \tilde W (\tau_{k})|} \Big)
	- \tilde V^{\eps}\Big(\tilde W (\tau_{k}), 
	\tau_{k+1},x(\tau_k)   \Big)}{{|\Delta \tilde W (\tau_{k})|}}
	-D_{\tilde x^i} \tilde V^{\eps}(\tilde W (\tau_{k}), 
	\tau_{k+1},x(\tau_k) )
\nonumber\\
&=
\int_0^1\int_0^1 
 D_{\tilde x^i\tilde x^i}\tilde V^{\eps}(\tilde W (\tau_{k}), 
	\tau_{k+1},x(\tau_k) + \lambda_1 \lambda e_i {|\Delta \tilde W (\tau_{k})|} ) \,\lambda d\lambda_1
  d\lambda \cdot |\Delta \tilde W (\tau_{k})|,\text{ a.s..}
  \label{defference-2}
\end{align}
}
Subtracting \eqref{defference-2} from \eqref{difference-1}  and  applying similarly the integration-by-parts formula yield 
{\small
\begin{align*}
&\left|  D_{\tilde x^i} \tilde V^{\eps}(\tilde W (\tau_{k}), 
	\tau_{k+1},x(\tau_k) )
- D_{\tilde x^i} \tilde V^{\eps}(\tilde W (\tau_{k}), 
	\tau_{k},x(\tau_k) ) \right|
\\
&= \bigg|
	\int_0^1\int_0^1 \Big( D_{\tilde x^i\tilde x^i} \tilde V^{\eps}(\tilde W (\tau_{k}), 
	\tau_{k},x(\tau_k) + \lambda_1 \lambda e_i {|\Delta \tilde W (\tau_{k})|} )
	\\
	&\quad\quad\quad\quad\quad\quad
	-
D_{\tilde x^i\tilde x^i} \tilde V^{\eps}(\tilde W (\tau_{k}), 
	\tau_{k+1},x(\tau_k) + \lambda_1 \lambda e_i {|\Delta \tilde W (\tau_{k})|} ) \Big) 
	\,\lambda d\lambda_1 d\lambda \cdot |\Delta \tilde W (\tau_{k})|
	\\
	&\quad
	+\int_0^1 \Big( D_t  \tilde V^{\eps}(\tilde W (\tau_{k}), 
	\tau_{k}+\lambda (\tau_{k+1}-\tau_k),x(\tau_k) +   e_i {|\Delta \tilde W (\tau_{k})|} )
	\\
	&\quad\quad\quad\quad\quad
	-D_t  \tilde V^{\eps}(\tilde W (\tau_{k}), 
	\tau_{k}+\lambda (\tau_{k+1}-\tau_k),x(\tau_k)   ) \Big)
		\,d\lambda \cdot |\tau_{k+1}-\tau_k| \cdot |\Delta \tilde W (\tau_{k})|^{-1}  \bigg|
\\
&
\leq
	C_2 \left(  |\tau_{k+1}-\tau_k|^{\frac{\bar \alpha}{2}}\cdot |\Delta \tilde W (\tau_{k})| + |\tau_{k+1}-\tau_k|\cdot  |\Delta \tilde W (\tau_{k})|^{\bar \alpha-1}  \right), \text{ a.s.,}
\end{align*}
}
which combined with \eqref{eps-1} yields that 
{\small
\begin{align*}
 &\sum_{k=0}^{\tilde N-1}
 \left( \nabla_i  V^{\eps}(\tau_{k+1},x_{\tau_k,\tau_{k+1}-\tau_k})- \nabla_i V^{\eps}(\tau_k,x_{\tau_k})\right) \int_{\tau_k}^{\tau_{k+1}} g(s)\,dW^l(s)
 \\
 &
 = \sum_{k=0}^{\tilde N-1} \left( \sum_{j=1}^{m_0}   D_{\tilde y^j\tilde x^i} \tilde V^{\eps} (\tilde W (\tau_{k})   , \tau_{k+1},x(\tau_k))  \cdot (\Delta \tilde W (\tau_{k}))^{j}  + \eps_k^2\right)
 \int_{\tau_k}^{\tau_{k+1}} g(s)\,dW^l(s)
 \\
 &
 =
 \sum_{k=0}^{\tilde N-1} \int_{\tau_k}^{\tau_{k+1}}   \sum_{j=1}^{m_0}  D_{\tilde y^j\tilde x^i}  \tilde V^{\eps} (\tilde W (\tau)   , \tau,x(\tau_k)) \, d\tilde W^j(\tau) 
 \int_{\tau_k}^{\tau_{k+1}} g(s)\,dW^l(s)
 \\
 &
 \quad
 + \sum_{k=0}^{\tilde N-1} \int_{\tau_k}^{\tau_{k+1}}   
  \left( 
  	  D_{\tilde y\tilde x^i} \tilde V^{\eps} (\tilde W (\tau)   , \tau,x(\tau_k))  
	-  D_{\tilde y\tilde x^i}  \tilde V^{\eps} (\tilde W (\tau_k)   , \tau_{k+1},x(\tau_k))
 			\right)'\, d\tilde W(\tau) 
 \\
 &\quad\quad \quad\quad
 \cdot \int_{\tau_k}^{\tau_{k+1}} g(s)\,dW^l(s)
 +\sum_{k=0}^{\tilde N-1}   \eps_k^2 
 \int_{\tau_k}^{\tau_{k+1}} g(s)\,dW^l(s),
\end{align*}
}
with 
$$
|\eps_k^2| 
\leq  
 C_3\left( |\Delta \tilde W (\tau_{k})|^{1+\bar \alpha} + |\tau_{k+1}-\tau_k|^{\frac{\bar \alpha}{2}}\cdot |\Delta \tilde W (\tau_{k})| + |\tau_{k+1}-\tau_k|\cdot  |\Delta \tilde W (\tau_{k})|^{\bar \alpha-1}  \right), \text{ a.s.,}
$$
where the constant $C_3$ is independent of the partition $\Pi$. 
As 
{\small
\begin{align*}
&E \sum_{k=0}^{\tilde N-1}\left| \left( |\Delta \tilde W (\tau_{k})|^{1+\bar \alpha} 
+ |\tau_{k+1}-\tau_k|^{\frac{\bar \alpha}{2}}\cdot |\Delta \tilde W (\tau_{k})|
+ |\tau_{k+1}-\tau_k|\cdot  |\Delta \tilde W (\tau_{k})|^{\bar \alpha-1}  \right)
 \int_{\tau_k}^{\tau_{k+1}} g(s)\,dW^l(s)\right |
 \\
 &
 \leq C \sum_{k=0}^{\tilde N-1} \left( E\left[|\Delta \tilde W (\tau_{k})|^{p+p\bar \alpha} 
 + |\tau_{k+1}-\tau_k|^{\frac{p\bar \alpha}{2}}\cdot |\Delta \tilde W (\tau_{k})|^p
 + |\tau_{k+1}-\tau_k|^p\cdot  |\Delta \tilde W (\tau_{k})|^{p\bar \alpha-p} \right] \right)^{\frac{1}{p}}
 \\
 &\quad\quad\quad\quad
 \cdot
 \left\{E \left[ \left( \int_{\tau_k}^{\tau_{k+1}} |g(s)|^2\,ds \right)^{\frac{p}{2(p-1)}} \right] \right\}^{\frac{p-1}{p}}
 \\
 &\leq
 C \sum_{k=0}^{\tilde N-1}  |\tau_{k+1}-\tau_k|^{1+\frac{\bar\alpha}{2}}
 \\
 &\rightarrow 0, \text{ as }|\Pi| \rightarrow 0,
\end{align*}
}
with $1<p<1/(1-\bar\alpha)$ and 
{\small
\begin{align*}
& E\left| \sum_{k=0}^{\tilde N-1} \int_{\tau_k}^{\tau_{k+1}}  
  \left( 
  	  D_{\tilde y\tilde x^i} \tilde V^{\eps} (\tilde W (\tau)   , \tau,x(\tau_k))  
	-   D_{\tilde y\tilde x^i} \tilde V^{\eps} (\tilde W (\tau_k)   , \tau_{k+1},x(\tau_k))
 			\right)'\, d\tilde W(\tau) 
 \int_{\tau_k}^{\tau_{k+1}} g(s)\,dW^l(s) \right|
 \\
 &
 \leq \sum_{k=0}^{\tilde N-1} 
 \left( E\int_{\tau_k}^{\tau_{k+1}}   
  	\left|   D_{\tilde y\tilde x^i}  \tilde V^{\eps} (\tilde W (\tau)   , \tau,x(\tau_k))  
	-   D_{\tilde y\tilde x^i} \tilde V^{\eps} (\tilde W (\tau_k)   , \tau_{k+1},x(\tau_k))
 			\right|^2 \, d\tau
 \cdot 
 E\int_{\tau_k}^{\tau_{k+1}} |g(s)|^2\,ds \right)^{1/2}
 \\
 &
 \leq C
 \sum_{k=0}^{\tilde N-1} 
 \left( E\int_{\tau_k}^{\tau_{k+1}}   
  	\left(\left| \tau_{k+1}-\tau\right|^{\bar \alpha} +  \left|W (\tau) - W (\tau_k) \right|^{2\bar\alpha}
	 \right) \, d\tau \cdot
 \left| \tau_{k+1}-\tau_k\right| \right)^{1/2}
 \\
 &\leq
 C
  \sum_{k=0}^{\tilde N-1}  |\tau_{k+1}-\tau_k|^{1+\frac{\bar\alpha}{2}}
 \\
 &
 \rightarrow 0, \text{ as }|\Pi| \rightarrow 0,
 \end{align*}
}
using standard computations for covariation (see \cite[Section 6 of Chapter II]{pP05} for instance) gives
{\small
\begin{align*}
&\lim_{|\Pi| \rightarrow 0^+} \sum_{k=0}^{\tilde N-1}  \left( \nabla_i V^{\eps}(\tau_{k+1},x_{\tau_k,\tau_{k+1}-\tau_k})- \nabla_i V^{\eps}(\tau_k,x_{\tau_k})\right) \int_{\tau_k}^{\tau_{k+1}} g(s)\,dW^l(s)
\\
&
=\lim_{|\Pi| \rightarrow 0^+}
 \sum_{k=0}^{\tilde N-1} \int_{\tau_k}^{\tau_{k+1}}   \sum_{j=1}^{m_0}   D_{\tilde y^j\tilde x^i} \tilde V^{\eps} (\tilde W (\tau)   , \tau,x(\tau_k)) \, d\tilde W^j(\tau) 
 \int_{\tau_k}^{\tau_{k+1}} g(s)\,dW^l(s)
 \\
 &=\int_{\tilde t_{N-1}}^{\tilde t_{N}} \left(\mathfrak{d}_{\omega} \nabla V^{\eps} \right)^{li} (t,x_t) g(t)\,dt, \text{ in probability.}
\end{align*}
 }
It follows similarly for subintervals $(t_{j-1},t_j)$ for $j=1,\ldots, N-1$ and this yields \eqref{dwD-value} and  $V^{\eps}  \in \mathscr C^2_{\sF}$.

In view of the approximations in Assumption $(\cA 3)$ and with an analogy to  (iv) in Proposition \ref{prop-value-func}, we may select the constant $\tilde{L}>0$ such that for all $t\in [0,T]\setminus\{t_0,\ldots,t_{N-1}\}$ with $x_t\in\Lambda_t$,
$
 |\nabla V^{\eps}(t,x_t)|   \
\leq \tilde L,\,\,\,\text{a.s.}
$, with $\tilde L$ being independent of $\eps$ and $N$. Put
\begin{align*}
\overline{V}^{\eps}(s,x)
=
	V^{\eps}(s,x_s)+Y^{\eps}(s),\quad
\underline{V}^{\eps}(s,x)
=
	V^{\eps}(s,x_s)-Y^{\eps}(s) .
\end{align*}
It remains to verify $\overline V^{\eps}\in \overline {\mathscr V}$ and $\underline V^{\eps}\in \underline {\mathscr V}$ and  find a constant $C_4$ independent of $\eps$ and $N$ s.t.
\begin{align*}
E\left| \overline V^{\eps}(s,x_s)-V(s,x_s)\right|
+E\left| \underline V^{\eps}(s,x_s)-V(s,x_s)\right|
\leq C_4 \cdot \eps, \quad \forall \, s\in [0,T] \text{ with }x_s\in \Lambda_s^0,
\end{align*}
which together with the relation $\overline V^{\eps}\geq V \geq \underline V^{\eps}$ finally yields $\underline u=V= \overline u$. As the remaining part of the proof is analogous to that of \cite[Theorem 5.6]{qiu2017viscosity}, it is omitted.
\end{proof}

In the above proof, we construct the approximations of $V$ with $\overline V^{\eps}$ and $\underline V^{\eps}$ that are just  $\tilde{\sF}_t$-adapted. Thus, for each $t\in[0,T]$ and $x_t\in\Lambda_t^0$, $V(t,x_t)$ is just $\tilde{\sF}_t$-measurable, which indicates that $\left(\mathfrak{d}_{\omega}\nabla V^{\eps}\right)^{l j}$ is actually vanishing for $l=m_0+1,\ldots, m$,  $j=1,\ldots,d$. Hence, under assumptions $(\cA1)-(\cA3)$, the  SPHJB equation may be equivalently written as
\begin{equation*}
  \left\{\begin{array}{l}
  \begin{split}
  -\mathfrak{d}_t u(t,x_t)
 & =\essinf_{v\in U} \bigg\{\text{tr}\left(\frac{1}{2}\sigma \sigma'(t,x_t,v) \nabla^2u(t,x_t)+\tilde\sigma(t,x_t,v) \mathfrak{d}_{\tilde\omega}\nabla u(t,x_t)\right)\\
 &\quad
+
\beta'(t,x_t,v)\nabla^2u(t,x_t) +f(t,x_t,v)
                \bigg\} 
,\, (t,x)\in [0,T)\times  C([0,T];\bR^d);\\
    u(T,x)&= G(x), \quad x\in   C([0,T];\bR^d),
    \end{split}
  \end{array}\right.
\end{equation*}
where we use the notation 
$ \mathfrak{d}_{\tilde\omega}\nabla u = \left(\left(  \mathfrak{d}_{\omega}\nabla u\right)^{kj}\right)_{1\leq k\leq m_0, 1\leq j\leq d}.$

\begin{appendix}

\end{appendix}
\bibliographystyle{siam}

\end{document}